\documentclass[10pt]{article}
\usepackage[utf8]{inputenc}
\usepackage[a4paper,margin=1.2in]{geometry}

\title{Non-existence and strong ill-posedness in $C^{k,\beta}$ for the generalized Surface Quasi-geostrophic equation}
\author{Diego C\'ordoba\footnote{dcg@icmat.es}\quad and Luis Mart\'inez-Zoroa\footnote{luis.martinez@icmat.es}\\ \\ \small Instituto de Ciencias Matem\'aticas CSIC-UAM-UCM-UC3M }

\usepackage{graphicx}
\usepackage{amsmath}
\usepackage{ dsfont }
\usepackage{amsfonts}
\usepackage{amsthm}

\newtheorem{theorem}{Theorem}[section]

\newtheorem{lemma}[theorem]{Lemma}
\newtheorem{remark}{Remark}

\begin{document}
\maketitle

\begin{abstract}
    We consider solutions to the generalized Surface Quasi-geostrophic equation ($\gamma$-SQG) when the velocity is more singular than the active scalar function (i.e. $\gamma\in(0,1)$). In this paper we establish strong ill-posedness in $C^{k,\beta}$ ($k\geq 1$, $\beta\in(0,1]$ and $k+\beta>1+\gamma$) and we also construct solutions  in $\mathds{R}^2$   that initially are in $C^{k,\beta}\cap L^2$ but are not in $C^{k,\beta}$ for $t>0$.  Furthermore these solutions stay in $H^{k+\beta+1-2\delta}$ for some small $\delta$ and an arbitrarily long time.
\end{abstract}

\section{Introduction}

In this paper we consider a family of active scalars in two dimensions, that are driven by an incompressible flow which is more singular than the scalar itself. More precisely, we say a function $w(x,t):\mathds{R}^2\times \mathds{R}_+\rightarrow \mathds{R}$, $w(x,t)\in H^{s}$, $s>2+\gamma$ is a solution to the generalized Surface Quasi-geostrophic equation ($\gamma$-SQG equation) with initial conditions $w(x,0)=w_{0}(x)$ if the equation

\begin{eqnarray}\label{gSQG}
\frac{\partial w}{\partial t} + v_{1,\gamma}\frac{\partial w}{\partial x_{1}} + v_{2,\gamma}\frac{\partial w}{\partial x_{2}}= 0
\end{eqnarray}
is fulfilled for every $x\in\mathds{R}^2$, with $v=(v_{1,\gamma},v_{2,\gamma})$ defined by
$$v_{1,\gamma}=-\frac{\partial}{\partial x_{2}}\Lambda^{-1+\gamma} w,\ v_{2,\gamma}=\frac{\partial}{\partial x_{1}}\Lambda^{-1+\gamma} w.$$  We denote $\Lambda^{\alpha} f\equiv (-\Delta)^{\frac{\alpha}{2}} f$ by the Fourier transform $\widehat{\Lambda^{\alpha} f}(\xi) = |\xi|^{\alpha}\widehat{f} (\xi)$. 

 This family of equations becomes the 2D incompressible Euler equations and the SQG equation (see \cite{Majda}, \cite{Chaewu} and \cite{Chaecordoba}) when $\gamma=-1,0$ respectively. For the entire range $\gamma\in(-1,1)$, it has been shown in \cite{Chaecordoba} that this system is locally well-possed in $H^{s}$ for $s>2+\gamma$.  In \cite{ChaeWu2} the authors proved local existence in the critical Sobolev space $H^2$ for a logarithmic inviscid regularization of SQG (see also \cite{Jolly} for the $\gamma$-SQG case). Regarding $H^s$ norm growth see \cite{Nazarov} where the authors show that there exists initial conditions with arbitrarily small $H^{s}$ norm ($s\geq 11$) that become large after a long period of time. Finite time formation of singularities for  initial data in $H^{s}$ for $s>2+\gamma$ remains an open problem for the range $\gamma\in(-1,1)$. On the other hand, there are a few rigorous constructions of non-trivial global solutions in $H^s$ (for some s satisfying $s> 2+\gamma$) in \cite{CCG2}, \cite{GS}, \cite{CCG} and \cite{We}.

 For both 2D Euler and SQG, the critical Sobolev space has been studied in  \cite{Bourgainsobolev}, \cite{Zoroacordoba}, \cite{Elgindisobolev} and \cite{Injee}, where it has been established non-existence of uniformly bounded solutions in $H^{2+\gamma}$ (see also  \cite{Kukavica} and \cite{Kwon} for other ill-posedness results for active scalars). Furthermore, for $\gamma=0$, in a range of supercritical Sobolev spaces ($s\in(\frac32,2)$) non-existence of solutions in $H^{s}$ is proved in \cite{Zoroacordoba}. 

Global existence of solutions in $L^2$ have already been obtained for SQG in \cite{Resnick}   (see \cite{Chaecordoba}, for an extension in the case $\gamma\in(0,1)$), but uniqueness is not known and in fact there is non uniqueness of solutions for  $\Lambda^{-1}w\in C_t^{\sigma}C_x^{\beta}$ with $\frac12<\beta<\frac45$ and $\sigma < \frac{\beta}{2-\beta}$ (see \cite{Buckmaster}).

Local well-posedness in $C^{k,\beta}\cap L^{q}$ ($k\geq 1$, $\beta\in(0,1)$, $q>1$) was established for SQG in \cite{Wu}, and recently the result was improved in \cite{Ambrose}, where  the requirement $w\in L^{q}$ has been dropped. The same result as in \cite{Wu} applies for the range $\gamma\in [-1,0)$ for $\beta\in[0,1]$ (for the a priori estimates see \cite{Chaewu}). Nevertheless, as shown in \cite{Zoroacordoba} for $\gamma=0$, there is no local existence result when $\beta=0,1$ (in the case of 2D Euler equations see \cite{Bourgaincm} and \cite{Elgindi} for a proof of strong ill-posedness and non-existence of uniformly bounded solutions for the velocity $v$ in $C^k$).

Global in time exponential growth of solutions was obtained in \cite{SmallscaleSQG} for the range $\gamma\in(-1,1)$ in $C^{1,\beta}$, with $\beta\in[f(\gamma),1]$.
\subsection{Main results}

The aim of this paper is to prove strong ill-posedness in $C^{k,\beta}$ ($k\geq 1$, $\beta\in(0,1]$ and $k+\beta>1+\gamma$) of the $\gamma$-SQG equation for the range  $\gamma\in(0,1)$. We also construct solutions  in $\mathds{R}^2$ of $\gamma$-SQG  that initially are in $C^{k,\beta}\cap L^2$ but are not in $C^{k,\beta}$ for $t>0$.

\begin{theorem} (Strong ill-posedness)
Given $k$ a natural number, $\beta\in(0,1]$, $\gamma\in(0,1)$ and $\delta\in(0,\frac12)$ with $k+\beta-2\delta> 1+\gamma$, then for any $T,t_{crit,}\epsilon_{1},\epsilon_{2}>0$, there exist a $H^{k+\beta+1-\delta}$ function $w(x,0)$ such that $||w(x,0)||_{C^{k,\beta}}\leq \epsilon_{1}$ and the only solution to (\ref{gSQG}) in $H^{k+\beta+1-\delta}$ with initial conditions $w(x,0)$ exists for $t\in[0,T]$ and fulfills that
$$||w(x,t_{crit})||_{C^{k,\beta}}\geq \frac{1}{\epsilon_{2}}.$$
\end{theorem}

\begin{theorem} (Non-existence)
Given $k$ a natural number, $\beta\in(0,1]$, $\gamma\in(0,1)$ and $\delta\in(0,\frac12)$ with $k+\beta-2\delta> 1+\gamma$, then for any $T$ and $\epsilon>0$, there exist a $H^{k+\beta+1-\frac32\delta}$ function $w(x,0)$ such that $||w(x,0)||_{C^{k,\beta}}\leq \epsilon$ and that the only solution to  (\ref{gSQG}) in $H^{k+\beta+1-\frac32 \delta}$  with initial conditions $w(x,0)$ exists for $t\in[0,T]$ and fulfills that, for $t\in(0,T]$, $||w(x,t)||_{C^{k,\beta}}=\infty.$
\end{theorem}

\begin{remark}
Although technically we do not prove the results for the case $\beta=0$, the results in $C^{k,1}$ actually gives us strong ill-posedness and non-existence in the space $C^{k+1}$.
\end{remark}

\subsection{Strategy of the proof}

To obtain the ill-posedness result, we first focus on finding a pseudo-solution $\bar{w}$ for $\gamma-SQG$ that exhibits the behaviour we would like to show, mainly that it has a small $C^{k,\beta}$ norm initially and this norm grows a lot in a very short period of time. We say that $\bar{w}$ is a pseudo-solution if it fulfils an evolution equation of the form

\begin{eqnarray}
\frac{\partial \bar{w}}{\partial t} + v_{1,\gamma}\frac{\partial \bar{w}}{\partial x_{1}} + v_{2,\gamma}\frac{\partial \bar{w}}{\partial x_{2}}+F(x,t)= 0
\end{eqnarray}
with $v=(v_{1,\gamma},v_{2,\gamma})$ defined by
$$v_{1,\gamma}=-\frac{\partial}{\partial x_{2}}\Lambda^{-1+\gamma} \bar{w},\ v_{2,\gamma}=\frac{\partial}{\partial x_{1}}\Lambda^{-1+\gamma} \bar{w}.$$
This, of course, is not a very restrictive definition, but in general we will only use this definition for $\bar{w}$ when $F$ is small in a relevant norm. Once we have a pseudo-solution $\bar{w}$ with the desired behaviour, if $F$ is small and  both $F$ and $\bar{w}$ are regular enough, then $\bar{w}\approx w$, with $w$ the solution to (\ref{gSQG}) with the same initial conditions as $\bar{w}$, and therefore $w$ shows the same fast growth as $\bar{w}$.

The details about how to find a pseudo-solution with the desired behaviour are somewhat technical, but the rough idea is to consider initial conditions that in polar coordinates have the form

$$w_{N}(r,\alpha,0)=f(r)+\frac{g(r,N\alpha)}{N^{k+\beta}},$$
that is, a radial function (which is a stationary solution to $\gamma$-SQG) plus a perturbation of frequency $N$ in $\alpha$. The evolution of $w_{pert,N}(r,\alpha,t):=w_{N}(r,\alpha,t)-f(r)$ satisfies

$$\frac{\partial w_{pert,N}}{\partial t}+v_{\gamma}(w_{pert,N})\cdot\nabla w_{pert,N}+v_{r,\gamma}(w_{pert,N})\frac{\partial f(r)}{\partial r}+\frac{\partial w_{pert,N}}{\partial \alpha}\frac{v_{\alpha,\gamma}(f(r))}{r}=0,$$
where $v_{r,\gamma}, v_{\alpha,\gamma}$ are the radial and angular components of the velocity respectively.

For very big $N$, we have that 
$$v_{\gamma}(w_{pert,N})\cdot\nabla w_{pert,N}\approx 0,\ v_{r,\gamma}(w_{pert,N})\approx C_{\gamma}(-\Delta_{\alpha})^{\frac{\gamma}{2}}H_{\alpha}(w_{pert,N})$$
where $(-\Delta_{\alpha})^{\frac{\gamma}{2}},H_{\alpha}$ are the fractional laplacian and the Hilbert transform respectively with respect to only the variable $\alpha.$ This suggest studying

\begin{equation}\label{1Dapprox}
   \frac{\partial \tilde{w}}{\partial t}+\frac{\partial f(r)}{\partial r}C_{\gamma}(-\Delta_{\alpha})^{\frac{\gamma}{2}}H_{\alpha}(\tilde{w})+\frac{\partial \tilde{w}}{\partial \alpha}\frac{v_{\alpha,\gamma}(f(r))}{r}=0. 
\end{equation}
and using $\bar{w}=f(r)+\tilde{w}$. The system (\ref{1Dapprox}) is relatively simple to study, since it is linear and one dimensional in nature, and one can obtain explicit solutions where the $C^{k,\beta}$ norm grows arbitrarily fast.  Then, once the candidate pseudo-solutions are found, a careful study of the errors involved allows us to obtain ill-posedness.

Moreover, to obtain non-existence, we consider an infinite number of fast growing solutions, and spread them through the plane so that the interactions between them become very small.

\subsection{Outline of the paper} 
The paper is organized as follows. In Section 2, we set the notation used through the paper. In
Section 3, we obtain estimates on the velocity in the radial and angular direction. In section 4, we introduce the pseudo-solutions with the desired properties and establish the necessary estimates on the source term $F(x,t)$. Finally in section 5, we prove strong ill-posedness and non-existence for the space $C^{k,\beta}$.

\section{Preliminaries and notation}


\subsection{Polar coordinates}\label{polarcoordinates}

Many of our computations and functions become much simpler if we use polar coordinates, so  we need to establish some notation in that regard. For the rest of this subsection, we will refer to
$$F:\mathds{R}_{+}\times[0,2\pi)\rightarrow \mathds{R}^2$$
$$(r,\alpha)\rightarrow (r\cos{(\alpha)},r\sin{(\alpha)})$$
the map from polar to cartesian coordinates. Note that the choice of $[0,2\pi)$ for the variable $\alpha$ is arbitrary and any interval of the form $[c,2\pi+c)$ would also work, and in fact we will sometimes consider intervals different from $[0,2\pi)$. These changes in the domain will not be specifically mentioned since they will be clear by context.

Given a function $f(x_{1},x_{2})$ from $\mathds{R}^{2}$ to $\mathds{R}$, we define

$$f^{pol}:\mathds{R}_{+}\times[0,2\pi)\rightarrow \mathds{R}$$
as  $f^{pol}(r,\alpha):=f(F(r,\alpha))$.

For $r>0$, we also have the following equalities
\begin{equation}\label{derx1}
    \frac{\partial f(x_{1},x_{2})}{\partial x_{1}}=\cos{(\alpha(x_{1},x_{2}))} \frac{\partial f^{pol}}{\partial r}(F^{-1}(x_{1},x_{2}))-\frac{1}{r}\sin{(\alpha(x_{1},x_{2}))}\frac{\partial f^{pol}}{\partial \alpha}(F^{-1}(x_{1},x_{2})),
\end{equation}

\begin{equation}\label{derx2}
    \frac{\partial f(x_{1},x_{2})}{\partial x_{2}}=\sin{(\alpha(x_{1},x_{2}))} \frac{\partial f^{pol}}{\partial r}(F^{-1}(x_{1},x_{2}))+\frac{1}{r}\cos{(\alpha(x_{1},x_{2}))}\frac{\partial f^{pol}}{\partial \alpha}(F^{-1}(x_{1},x_{2})).
\end{equation}

Furthermore, for functions such that $supp(f^{pol}(r,\alpha))\subset\{(r,\alpha):r\geq r_{0}\}$ with $r_{0}>0$, we have that for $m=0,1,...$, using (\ref{derx1}) and (\ref{derx2})

$$||f||_{C^{m}}\leq C_{r_{0},m}||f^{pol}||_{C^{m}},$$
where

$$||f^{pol}||_{C^{m}}=\sum_{k=0}^{m}\sum_{i=0}^{k}||\frac{\partial^{k}f^{pol}}{\partial r^{i}\partial \alpha^{k-i}}||_{L^{\infty}},$$
and similarly
\begin{equation}\label{equivcmb}
    ||f||_{C^{m,\beta}}\leq C_{r_{0},m,\beta}||f^{pol}||_{C^{m,\beta}}.
\end{equation}
with
\begin{align*}
    &||f^{pol}(r,\alpha)||_{C^{m,\beta}}=||f^{pol}||_{C^{m}}\\
    &+\sum_{i=0}^{k}sup_{R,\in[0,\infty],A\in[0,2\pi],h_{1}\in[-R,\infty],h_{2}\in[-\pi,\pi]}\frac{|\frac{\partial^{m} f^{pol}}{\partial^{i}r\partial^{m-i} \alpha}(R,A)-\frac{\partial^{m} f}{\partial^{i}r \partial^{m-i} \alpha}(R+h_{1},A+h_{2})}{|h_{1}^{2}+h_{2}^{2}|^{\frac{\beta}{2}}}.\\
\end{align*}

Furthermore, if we restrict ourselves to functions such that $supp(f^{pol}(r,\alpha))\subset\{(r,\alpha):r_{1}\geq r\geq r_{0}\}$ with $r_{1}>r_{0}>0$ then for $m=0,1,...$

$$||f||_{H^{m}}\leq C_{r_{1},r_{0},m}||f^{pol}||_{H^{m}},$$
with

$$||f^{pol}||_{H^{m}}=\sum_{k=0}^{m}\sum_{i=0}^{k}||\frac{\partial^{k}f^{pol}}{\partial r^{i}\partial \alpha^{k-i}}||_{L^{2}}.$$

Since we will need to compute integrals in polar coordinates, for a general set $S$ we will use the notation

$$S^{pol}:=\{(r,\alpha):F(r,\alpha)\in S\}$$
and more specifically, we will use

$$B_{\lambda}^{pol}(R,A):=\{(r,\alpha):|F(r,\alpha)-F(R,A)|\leq \lambda\}$$
with $|(x_{1},x_{2})|=|x_{1}^{2}+x_{2}^{2}|^{\frac12}$ (this is simply the set $B_{\lambda}(R\cos{(A)},R\sin{(A)})$ in polar coordinates). Also, note that, for $R\geq 2\lambda$ (which we will assume from now on) we have
$$B_{\lambda}^{pol}(R,A)\subset [R-\lambda,R+\lambda]\times [A-\arccos{(1-\frac{\lambda^2}{R^2})},A+\arccos{(1-\frac{\lambda^2}{R^2})}].$$

We also define, for $h\in[-\lambda,\lambda]$,

$$S_{\lambda,R,A}(h):=sup(\tilde{\alpha}:(R+h,A+\tilde{\alpha})\in B^{pol}_{\lambda}(R,A))$$
and  defining
$$S_{\lambda,R,A,\infty}:=sup_{h\in[-\lambda,\lambda]}(S_{\lambda,R,A}(h))$$
then for $\tilde{\alpha}\in[-S_{\lambda,R,A,\infty},S_{\lambda,R,A,\infty}]$ we can define

$$P_{\lambda,R,A,+}(\tilde{\alpha}):=sup(h:(R+h,A+\tilde{\alpha})\in B^{pol}_{\lambda}(R,A))$$
$$P_{\lambda,R,A,-}(\tilde{\alpha}):=inf(h:(R+h,A+\tilde{\alpha})\in B^{pol}_{\lambda}(R,A)).$$

When the values of $\lambda, R$ and $A$ are clear by context, we will just write $S(h),S_{\infty},P_{+}(\tilde{\alpha})$ and $P_{-}(\tilde{\alpha})$. A property for $P_{+}(\tilde{\alpha})$ and $P_{-}(\tilde{\alpha})$ that we will need to use later on is that, for $R\in[\frac12,\frac32]$ and $\tilde{\alpha}\in[-S_{\lambda,R,A,\infty},S_{\lambda,R,A,\infty}]$ we have

    $$|P_{\lambda,R,A,+}(\tilde{\alpha})+P_{\lambda,R,A,-}(\tilde{\alpha})|\leq C \lambda^2.$$

Which can be easily obtained using that, since 
$$|F(R,A)-F(r,\alpha)|=|(R-r)^2+2Rr(1-\cos(A-\alpha))|^{\frac12}$$
then
$$P_{\lambda,R,A,+}(\tilde{\alpha})=\frac{-2R(1-\cos(\tilde{\alpha}))+\sqrt{(2R(1-\cos(\tilde{\alpha}))^2-4(2R^2(1-\cos(\tilde{\alpha}))-\lambda^2)}}{2}$$
$$P_{\lambda,R,A,-}(\tilde{\alpha})=\frac{-2R(1-\cos(\tilde{\alpha}))-\sqrt{(2R(1-\cos(\tilde{\alpha}))^2-4(2R^2(1-\cos(\tilde{\alpha}))-\lambda^2)}}{2},$$
so

$$|P_{\lambda,R,A,+}(\tilde{\alpha})+P_{\lambda,R,A,-}(\tilde{\alpha})|=4R(1-\cos(\tilde{\alpha}))\leq C\tilde{\alpha}^2\leq C \lambda^2.$$


\subsection{Other notation}

Given two sets  $X,Y\subset\mathds{R}^2$, we will use $d(X,Y)$ to refer to the distance between the two, that is

$$d(X,Y):=\inf_{x\in X,y\in Y} |x-y|=\inf_{x\in X,y\in Y}|(x_{1}-y_{1})^2+(x_{2}-y_{2})^2|^{\frac12}.$$

Furthermore, given a function $f$ and a set $X$ we define $d(f,X)$ as $d(supp(f),X).$
Also, given a set $X$ and a point $x$ we define the set

$$X-x:=\{y\in\mathds{R}^2: y+x\in X\}.$$

Working in polar coordinates, we will use the notation

$$X^{pol}-(r,\alpha):=\{(\tilde{r},\tilde{\alpha})\in \mathds{R}^2:(\tilde{r}+r,\tilde{\alpha}+\alpha)\in X^{pol}\},$$
where we need to be careful since $X^{pol}-(r,\alpha)\neq (X-F(r,\alpha))^{pol}$.

We will also define, for $A$ a regular enough set, $k\in\mathds{N}$

$$||f(x) 1_{A}||_{C^{k}}:=\sum_{i=0}^{k}\sum_{j=0}^{i}\text{ess-sup}_{x\in A}(\frac{\partial^{i}f(x)}{\partial^{j}x_{1}\partial^{i-j} x_{2}}),$$

$$||f(x)1_{A}||_{H^{k}}:=\sum_{i=0}^{s}\sum_{j=0}^{i} (\int_{A} (\frac{\partial^{i}f(x)}{\partial^{j}x_{1}\partial^{i-j} x_{2}})^2 dx)^{\frac12}.$$

Finally we will use the notation

$$|f|_{C^{k,\beta}}:=\sum_{i=0}^{k}sup_{h_{1},h_{2}\in \mathds{R}}\frac{|\frac{\partial^{k} f}{\partial^{i}x_{1}\partial^{k-i}\partial x_{2}}(y_{1},y_{2})-\frac{\partial^{k} f}{\partial^{i}x_{1}\partial^{k-i}\partial x_{2}}(y_{1}+h_{1},y_{2}+h_{2})|}{|h_{1}^{2}+h_{2}^{2}|^{\frac{\beta}{2}}}.$$

\subsection{The velocity}
We will be considering $\gamma$-SQG, so our scalar $w$ will be transported with a velocity given by

$$v_{\gamma}(w(.))(x)=C(\gamma)P.V. \int_{\mathds{R}^2} \frac{(x-y)^{\perp}w(y)}{|x-y|^{(3+\gamma)/2}}dy_{1}dy_{2}.$$

Since the results are independent of the specific value of $C(\gamma)$, we will just assume $C(\gamma)=1$. Furthermore we will use the notation

$$v_{1,\gamma}(w(.))(x)=v_{\gamma}\cdot(1,0)=P.V. \int_{\mathds{R}^2} \frac{(y_{2}-x_{2})w(y)}{|x-y|^{(3+\gamma)/2}}dy_{1}dy_{2},$$
$$v_{2,\gamma}(w(.))(x)=v_{\gamma}\cdot(0,1)=P.V. \int_{\mathds{R}^2} \frac{(x_{1}-y_{1})w(y)}{|x-y|^{(3+\gamma)/2}}dy_{1}dy_{2}.$$

The operators $v_{\gamma},v_{1,\gamma}$ and $v_{2,\gamma}$ have several useful properties that we will be using later, namely the fact that they commute with cartesian derivatives $\frac{\partial}{\partial x_{1}}$ and $\frac{\partial }{\partial x_{2}}$ (as long as $w$ is regular enough) and also that, for $i=1,2$

$$||v_{i,\gamma}(w)||_{H^{k}}\leq C_{k,\gamma}||w||_{H^{k+\gamma}}.$$

It is unclear (and in fact, untrue) whether these properties translate to the operators $v_{r,\gamma}$ and $v_{\alpha,\gamma}$ that give us the velocity in the radial and polar direction respectively. We can obtain, however, similar properties  for these operators.

We start by noting that

\begin{align}\label{vrv1v2}
    &v_{r,\gamma}(w)=\cos(\alpha(x)) v_{1,\gamma}(w)+\sin(\alpha(x))v_{2,\gamma}(w),\\\nonumber
    &\\\nonumber
&v_{\alpha,\gamma}(w)=\cos(\alpha(x))v_{2,\gamma}(w)-\sin(\alpha(x))v_{1,\gamma}(w),\\ \nonumber
\end{align}

and since $\cos(\alpha(x))$ and $\sin(\alpha(x))$ are $C^{\infty}$ if we are not close to $r=0$, we have that, for $m\in\mathds{Z}$

$$||v_{r,\gamma}(w) 1_{|x|\geq \frac12}||_{H^{m}}\leq C_{m}( ||v_{1,\gamma}(w)||_{H^{m}}+||v_{2,\gamma}(w)||_{H^{m}})\leq C_{m,\gamma}||w||_{H^{m+\gamma}},$$

$$||v_{\alpha,\gamma}(w) 1_{|x|\geq \frac12}||_{H^{m}}\leq C_{m}( ||v_{1,\gamma}(w)||_{H^{m}}+||v_{2,\gamma}(w)||_{H^{m}})\leq C_{m,\gamma}||w||_{H^{m+\gamma}}.$$

Furhtermore, if we differentiate with respect to $\frac{\partial}{\partial x_{i}}$, $i=1,2$ we get

$$\frac{\partial v_{r,\gamma}(w)}{\partial x_{i}}= v_{r,\gamma}(\frac{\partial w}{\partial x_{i}})+\frac{\partial \cos(\alpha(x))}{\partial x_{i}} v_{1,\gamma}(w)+\frac{\partial \sin(\alpha(x))}{\partial x_{i}}v_{2,\gamma}(w),$$

$$\frac{\partial v_{\alpha,\gamma}(w)}{\partial x_{i}}=v_{\alpha,\gamma}(\frac{\partial w}{\partial x_{i}})+\frac{\partial \cos(\alpha(x))}{\partial x_{i}}v_{2,\gamma}(w)-\frac{\partial \sin(\alpha(x))}{\partial x_{i}}v_{1,\gamma}(w).$$

With this, using induction and if we only consider $|x|\geq \frac12$ we get that, for $m_{1},m_{2}\in \mathds{Z}$

\begin{align*}
    &|\frac{\partial^{m_{1}+m_{2}} v_{r,\gamma}(w)}{\partial x_{1}^{m_{1}}\partial x_{2}^{m_{2}}}(x)- v_{r,\gamma}(\frac{\partial^{m_{1}+m_{2}} w}{\partial x_{1}^{m_{1}}\partial x_{2}^{m_{2}}})(x)|\\
    &\leq C \sum_{k=0}^{m_{1}+m_{2}-1} \sum_{j=0}^{k} |\frac{\partial^{k} v_{1,\gamma}(w)}{\partial x_{1}^{j}\partial x_{2}^{k-j}}(x)| +|\frac{\partial^{k} v_{2,\gamma}(w)}{\partial x_{1}^{j}\partial x_{2}^{k-j}}(x)|,\\
\end{align*}

\begin{align*}
    &|\frac{\partial^{m_{1}+m_{2}} v_{\alpha,\gamma}(w)}{\partial x_{1}^{m_{1}}\partial x_{2}^{m_{2}}}(x)- v_{\alpha,\gamma}(\frac{\partial^{m_{1}+m_{2}} w}{\partial x_{1}^{m_{1}}\partial x_{2}^{m_{2}}})(x)|\\
    &\leq C \sum_{k=0}^{m_{1}+m_{2}-1} \sum_{j=0}^{k} |\frac{\partial^{k} v_{1,\gamma}(w)}{\partial x_{1}^{j}\partial x_{2}^{k-j}}(x)| +|\frac{\partial^{k} v_{2,\gamma}(w)}{\partial x_{1}^{j}\partial x_{2}^{k-j}}(x)|,\\
\end{align*}
and thus
\begin{align}\label{dercomr}
    &||(\frac{\partial^{m_{1}+m_{2}} v_{r,\gamma}(w)}{\partial x_{1}^{m_{1}}\partial x_{2}^{m_{2}}}- v_{r,\gamma}(\frac{\partial^{m_{1}+m_{2}} w}{\partial x_{1}^{m_{1}}\partial x_{2}^{m_{2}}}))1_{|x|\geq \frac12}||_{L^{\infty}}\\\nonumber
    &\leq C (||v_{1,\gamma}(w) 1_{|x|\geq \frac12}||_{C^{m_{1}+m_{2}-1}}+||v_{2,\gamma}(w) 1_{|x|\geq \frac12}||_{C^{m_{1}+m_{2}-1}}),\\\nonumber
    &||(\frac{\partial^{m_{1}+m_{2}} v_{r,\gamma}(w)}{\partial x_{1}^{m_{1}}\partial x_{2}^{m_{2}}}- v_{r,\gamma}(\frac{\partial^{m_{1}+m_{2}} w}{\partial x_{1}^{m_{1}}\partial x_{2}^{m_{2}}}))1_{|x|\geq \frac12}||_{L^{2}}\\\nonumber
    &\leq C (||v_{1,\gamma}(w) 1_{|x|\geq \frac12}||_{H^{m_{1}+m_{2}-1}}+||v_{2,\gamma}(w) 1_{|x|\geq \frac12}||_{H^{m_{1}+m_{2}-1}})\\\nonumber
\end{align}

\begin{align}\label{dercomalpha}
    &||(\frac{\partial^{m_{1}+m_{2}} v_{\alpha,\gamma}(w)}{\partial x_{1}^{m_{1}}\partial x_{2}^{m_{2}}}- v_{\alpha,\gamma}(\frac{\partial^{m_{1}+m_{2}} w}{\partial x_{1}^{m_{1}}\partial x_{2}^{m_{2}}}))1_{|x|\geq \frac12}||_{L^{\infty}}\\\nonumber
    &\leq C (||v_{1,\gamma}(w) 1_{|x|\geq \frac12}||_{C^{m_{1}+m_{2}-1}}+||v_{2,\gamma}(w) 1_{|x|\geq \frac12}||_{C^{m_{1}+m_{2}-1}});\\\nonumber
    &||(\frac{\partial^{m_{1}+m_{2}} v_{\alpha,\gamma}(w)}{\partial x_{1}^{m_{1}}\partial x_{2}^{m_{2}}}- v_{\alpha,\gamma}(\frac{\partial^{m_{1}+m_{2}} w}{\partial x_{1}^{m_{1}}\partial x_{2}^{m_{2}}}))1_{|x|\geq \frac12}||_{L^{2}}\\\nonumber
    &\leq C (||v_{1,\gamma}(w) 1_{|x|\geq \frac12}||_{H^{m_{1}+m_{2}-1}}+||v_{2,\gamma}(w) 1_{|x|\geq \frac12}||_{H^{m_{1}+m_{2}-1}})\\\nonumber
\end{align}


with $C$ depending on $m_{1}$ and $m_{2}$. 

\section{Bounds for the velocity}

Since we will work in polar coordinates, it will be necessary to obtain expressions for the velocity in the radial and angular direction. These expressions are, assuming $w(x)$ is a $C^1$ function with compact support,

$$v^{pol}_{r,\gamma}(w)(r,\alpha)=\int_{[-r,\infty]\times[-\pi,\pi]}\frac{(r+h)^2\sin(\alpha')(w^{pol}(r+h,\alpha'+\alpha)-w^{pol}(r,\alpha))}{|h^2+2r(r+h)(1-\cos(\alpha'))|^{(3+\gamma)/2}}d\alpha'dh$$

$$v^{pol}_{\alpha,\gamma}(w)(r,\alpha)=\int_{[-r,\infty]\times[-\pi,\pi]}\frac{(r+h)(r-(r+h)\cos(\alpha'))(w^{pol}(r+h,\alpha'+\alpha)-w^{pol}(r,\alpha))}{|h^2+2r(r+h)(1-\cos(\alpha'))|^{(3+\gamma)/2}}d\alpha'dh.$$

These expressions, however, hide some cancellation of the kernel when we are far from the support of $w$. Therefore, given a $C^1$ function $w$ with support in $B_{\lambda}(R\cos(A),R\sin(A))$, $\frac32>R>\frac12$, $\lambda\leq \frac{1}{100}$ we will use the expressions

$$v^{pol}_{r,\gamma}(w)(r,\alpha)=\int_{B^{pol}_{4\lambda}(r,\alpha)-(r,\alpha)}\frac{(r+h)^2\sin(\alpha')(w^{pol}(r+h,\alpha'+\alpha)-w^{pol}(r,\alpha))}{|h^2+2r(r+h)(1-\cos(\alpha'))|^{(3+\gamma)/2}}d\alpha'dh$$

$$v^{pol}_{\alpha,\gamma}(w)(r,\alpha)=\int_{B^{pol}_{4\lambda}(r,\alpha)-(r,\alpha)}\frac{(r+h)(r-(r+h)\cos(\alpha'))(w^{pol}(r+h,\alpha'+\alpha)-w^{pol}(r,\alpha))}{|h^2+2r(r+h)(1-\cos(\alpha'))|^{(3+\gamma)/2}}d\alpha'dh$$
when $(r,\alpha)\in B_{2\lambda}(R,A)$ and

$$v^{pol}_{r,\gamma}(w)(r,\alpha)=\int_{supp(w^{pol})-(r,\alpha)}\frac{(r+h)^2\sin(\alpha')w^{pol}(r+h,\alpha'+\alpha)}{|h^2+2r(r+h)(1-\cos(\alpha'))|^{(3+\gamma)/2}}d\alpha'dh$$

$$v^{pol}_{\alpha,\gamma}(w)(r,\alpha)=\int_{supp(w^{pol})-(r,\alpha)}\frac{(r+h)(r-(r+h)\cos(\alpha'))w^{pol}(r+h,\alpha'+\alpha)}{|h^2+2r(r+h)(1-\cos(\alpha'))|^{(3+\gamma)/2}}d\alpha'dh$$
when $(r,\alpha)\notin B_{2\lambda}(R,A)$.

Although the expression for $B^{pol}_{4\lambda}(r,\alpha)$ is not simple, it will be enough for our computations to use the properties we obtained in subsection \ref{polarcoordinates}.


We are particularly interested in obtaining the velocity produced by $w$ with support very concentrated around some point far from $r=0$ (say $r=1$ for simplicity), and for this we start with the following technical lemma.

\begin{lemma}\label{aproxkernel}
Given $\lambda\leq \frac{1}{100},$ and a $C^1$ function $w(x)$ with $supp(w)\subset B_{\lambda}(\cos(c), \sin(c))$, $c\in\mathds{R}$, $\lambda\leq \frac{1}{100}$ we have that if $(r,\alpha)\in B_{2\lambda}(\cos(c),\sin(c))$ then

$$|v^{pol}_{r,\gamma}(w)(r,\alpha)-\int_{B^{pol}_{4\lambda}(r,\alpha)-(r,\alpha)}\frac{r^2\alpha'(w^{pol}(r+h,\alpha'+\alpha)-w^{pol}(r,\alpha))}{|h^2+r^2(\alpha')^2|^{(3+\gamma)/2}}d\alpha'dh|\leq C ||w||_{L^{\infty}}\lambda^{1-\gamma},$$

$$|v^{pol}_{\alpha,\gamma}(w)(r,\alpha)+\int_{B^{pol}_{4\lambda}(r,\alpha)-(r,\alpha)}\frac{rh(w^{pol}(r+h,\alpha'+\alpha)-w^{pol}(r,\alpha))}{|h^2+r^2(\alpha')^2|^{(3+\gamma)/2}}d\alpha'dh|\leq C ||w||_{L^{\infty}}\lambda^{1-\gamma},$$



with $C$ depending on $\gamma$.
\end{lemma}

\begin{remark}
The result can be extended to functions with support concentrated around a point $(r,\alpha)$ with $r\neq 0$, although then the constant will depend on the specific value of $r$.
\end{remark}

\begin{proof}
This result is very similar to  lemma 2.1 in \cite{Zoroacordoba}, and the proof is analogous. We just need to take successive approximations of the kernel and bound the error produced by each such approximation. For example, for $(r,\alpha)\in B_{2\lambda}(\cos(c),\sin(c))$ we have that

$$|\int_{B^{pol}_{4\lambda}(r,\alpha)-(r,\alpha)}(r+h)^2\frac{\sin(\alpha')-\alpha'}{|h^2+2r(r+h)(1-\cos(\alpha'))|^{(3+\gamma)/2}}(w^{pol}(r+h,\alpha'+\alpha)-w^{pol}(r,\alpha))d\alpha'dh|$$

$$\leq |\int_{B^{pol}_{4\lambda}(r,\alpha)-(r,\alpha)}(r+h)^2\frac{|\alpha'|^3(w^{pol}(r+h,\alpha'+\alpha)-w^{pol}(r,\alpha))}{|h^2+2r(r+h)(1-\cos(\alpha'))|^{(3+\gamma)/2}}d\alpha'dh|\leq C \lambda^{2-\gamma}||w||_{L^{\infty}}$$

and thus we can substitute the $\sin(\alpha'-\alpha)$ by $\alpha'-\alpha$ with an error small enough for our bounds. Repeating this process for other parts of the kernel yields the desired result.
\end{proof}

\begin{lemma}\label{aproxfuncion}
Given a natural number $N$, $\frac12>\delta>0$ fulfilling $N^{-\delta}\leq \frac{1}{100}$ and $N^{-1+\delta}<\frac{1}{100}$, a function $f_{N,\delta}(x)$ with $supp(f_{N,\delta})\subset B_{N^{-1+\delta}}(\cos(c_{1}),\sin(c_{1}))$ ($c_{1}\in\mathds{R}$), $||f_{N,\delta}^{pol}||_{C^j}\leq MN^{j(1-\delta)}$ for $j=0,1,2$ and $1>\gamma>0,$ then if $w_{N,\delta}^{pol}(r,\alpha):=f_{N,\delta}^{pol}(r,\alpha)\cos(N\alpha+c_{2})$ ($c_{2}\in\mathds{R}$) we have that for $(r,\alpha)\in B^{pol}_{2N^{-1+\delta}}(1,c_{1})$

\begin{align*}
    &\bigg|\int_{B^{pol}_{4N^{-1+\delta}}(r,\alpha)-(r,\alpha)}\frac{r^2\alpha'(w^{pol}(r+h,\alpha'+\alpha)-w^{pol}(r,\alpha))}{|h^2+r^2(\alpha')^2|^{(3+\gamma)/2}}d\alpha'dh\\
    &-f_{N,\delta}^{pol}(r,\alpha)\int_{B^{pol}_{4N^{-1+\delta}}(r,\alpha)-(r,\alpha)}\frac{r^2\alpha'(\cos(N(\alpha'+\alpha)+c_{2})-\cos(N\alpha+c_{2}))}{|h^2+r^2(\alpha')^2|^{(3+\gamma)/2}}d\alpha'dh\bigg|\leq CMN^{\gamma-\delta} \\
\end{align*}

\begin{align}\label{1ºvalpha}
    &\bigg|\int_{B^{pol}_{4N^{-1+\delta}}(r,\alpha)-(r,\alpha)}\frac{rh(w^{pol}(r+h,\alpha'+\alpha)-w^{pol}(r,\alpha))}{|h^2+r^2(\alpha')^2|^{(3+\gamma)/2}}d\alpha'dh\\\nonumber
    &-f_{N,\delta}^{pol}(r,\alpha)\int_{B^{pol}_{4N^{-1+\delta}}(r,\alpha)-(r,\alpha)}\frac{rh(\cos(N(\alpha'+\alpha)+c_{2})-\cos(N\alpha+c_{2}))}{|h^2+r^2(\alpha')^2|^{(3+\gamma)/2}}d\alpha'dh\bigg|\leq CMN^{\gamma-\delta} \\\nonumber
\end{align}
with $C$ depending on $\gamma$ and $\delta$.
\end{lemma}

\begin{proof}
We will just consider the case $c_{1},c_{2}=0$ for simplicity, and we will focus on obtaining (\ref{1ºvalpha}), the other inequality being analogous.
We need to find bounds for

$$\bigg|\int_{B^{pol}_{4N^{-1+\delta}}(r,\alpha)-(r,\alpha)}\frac{rh(w_{N,\delta}^{pol}(r+h,\alpha'+\alpha)-f_{N,\delta}^{pol}(r,\alpha)\cos(N(\alpha'+\alpha)))}{|h^2+r^2(\alpha')^2|^{(3+\gamma)/2}}d\alpha'dh\bigg|$$
$$=\bigg|\int_{-4N^{-1+\delta}}^{4N^{-1+\delta}}\int_{-S(h)}^{S(h)}\frac{rh(f_{N,\delta}^{pol}(r+h,\alpha'+\alpha)-f_{N,\delta}^{pol}(r,\alpha))\cos(N(\alpha'+\alpha))}{|h^2+r^2(\alpha')^2|^{(3+\gamma)/2}}d\alpha'dh\bigg|$$
$$=\bigg|\int_{-4N^{-1+\delta}}^{4N^{-1+\delta}}\int_{-rS(s_{2})}^{rS(s_{2})}\frac{s_{2}(f_{N,\delta}^{pol}(r+s_{2},\frac{s_{1}}{r}+\alpha)-f_{N,\delta}^{pol}(r,\alpha))\cos(N(\frac{s_{1}}{r}+\alpha))}{|s|^{3+\gamma}}ds_{1}ds_{2}\bigg|$$

where 
we used the change of variables $s_{1}=r(\alpha'-\alpha)$, $h=s_{2}$ and we define $|s|:=|s_{1}^2+s_{2}^2|^{\frac12}$. Furthermore, 

$$\int_{-4N^{-1+\delta}}^{4N^{-1+\delta}}\int_{-rS(s_{2})}^{rS(s_{2})}\frac{s_{2}\cos(N(\frac{s_{1}}{r}+\alpha))(f_{N,\delta}^{pol}(r+s_{2},\frac{s_{1}}{r}+\alpha)-f_{N,\delta}^{pol}(r,\alpha))}{|s|^{3+\gamma}}ds_{1}ds_{2}$$

$$=\cos(N\alpha)\int_{-4N^{-1+\delta}}^{4N^{-1+\delta}}\int_{-rS(s_{2})}^{rS(s_{2})}\frac{s_{2}\cos(\frac{N}{r}s_{1})(f_{N,\delta}^{pol}(r+s_{2},\frac{s_{1}}{r}+\alpha)-f_{N,\delta}^{pol}(r,\alpha))}{|s|^{3+\gamma}}ds_{1}ds_{2}$$
$$-\sin(N\alpha)\int_{-4N^{-1+\delta}}^{4N^{-1+\delta}}\int_{-rS(s_{2})}^{rS(s_{2})}\frac{s_{2}\sin(\frac{N}{r}s_{1})(f_{N,\delta}^{pol}(r+s_{2},\frac{s_{1}}{r}+\alpha)-f_{N,\delta}^{pol}(r,\alpha))}{|s|^{3+\gamma}}ds_{1}ds_{2}.$$

We will only check the term that is multiplied by $\cos(N\alpha)$, the other term being analogous. We start with the contribution when $(s_{1},s_{2})\in \mathcal{A}:=\{|s_{j}|\leq  \frac{4\pi r}{N}$ with $j=1,2$\}, which gives us

$$|\int_{\mathcal{A}}\frac{s_{2}\cos(\frac{N}{r}s_{1})(f_{N,\delta}^{pol}(r+s_{2},\frac{s_{1}}{r}+\alpha)-f_{N,\delta}^{pol}(r,\alpha))}{|s|^{3+\gamma}}ds_{1}ds_{2}|$$
$$\leq CMN^{\gamma-\delta}.$$

Next we consider the integral in
$$\mathcal{B}:= \{ (s_{1},s_{2}):(s_{1},s_{2})\in B^{pol}_{4N^{-1+\delta}}(r,\alpha)-(r,\alpha),|s_{1}|\leq  \lfloor\frac{S(s_{2})N}{2\pi}\rfloor \frac{2\pi r}{N}\}\setminus \mathcal{A},$$ with $\lfloor \cdot \rfloor$ the integer part.


We will focus on the contribution when $(s_{1},s_{2})\in \mathcal{B}\cap (s_{1}\geq \frac{4\pi r}{N},s_{2}\geq 0)$, since the other parts of the integral are bounded analogously. We start by computing the integral with respect to $s_{1}$.

For this we first note that, for an integer  $i$, given a $C^{2}$ function $g(x)$ and a real number $\frac{N}{r}>0$ we have

$$|\int_{i \frac{2\pi r}{N}}^{(i+1)\frac{2\pi r}{N}} \cos(\frac{N}{r}x)g(x)dx|\leq  (\frac{\pi r}{N})^3 (sup_{x\in(i\frac{2\pi r}{N},(i+1)\frac{2\pi r}{N})} |g''(x)|) $$
where $g''(x)$ is the second derivative of $g(x)$. This bound is obtained simply by considering a second order Taylor expansion around the middle point of the interval and noting that the constant and linear terms vanish. Therefore, if $i\geq 2$, $s_{2}>0$

\begin{align*}
&\bigg|\int_{i\frac{2\pi r}{N}}^{(i+1)\frac{2\pi r}{N}}\frac{\cos(\frac{N}{r}s_{1})(f_{N,\delta}^{pol}(r+s_{2},\frac{s_{1}}{r}+\alpha)-f_{N,\delta}^{pol}(r,\alpha))}{|s|^{3+\gamma}}ds_{1}\bigg|\\
&\leq (\frac{2\pi r}{N})^3 (sup_{s_{1}\in(i\frac{2\pi r}{N},(i+1)\frac{2\pi r}{N})}\bigg|\frac{d^2}{d s_{1}^2} \frac{f_{N,\delta}^{pol}(r+s_{2},\frac{s_{1}}{r}+\alpha)-f_{N,\delta}^{pol}(r,\alpha)}{|s|^{3+\gamma}}\bigg|)\\
&\leq CM(\frac{2\pi r}{N})^3\frac{1}{((\frac{i2\pi r}{N})^2+s_{2}^2)^{\frac{3+\gamma}{2}}}\Big(N^{2-2\delta}+\frac{N^{1-\delta}}{((\frac{i2\pi r}{N})^2+s_{2}^2)^{\frac12}}+\frac{N^{1-\delta}[(i+1)\frac{2\pi }{N}+s_{2}]}{((\frac{i2\pi r}{N})^2+s_{2}^2)}\Big)\\
&\leq CM(\frac{2\pi r}{N})^3\frac{1}{(\frac{i2\pi r}{N }+s_{2})^{3+\gamma}}\Big(N^{2-2\delta}+\frac{N^{1-\delta}}{(\frac{i2\pi r}{N}+s_{2})}\Big).\\
\end{align*}

Adding over all the relevant values of $i$ we get

\begin{align*}
  &\sum_{i=2}^{\lfloor \frac{S(s_{2})N}{2\pi}\rfloor}CM(\frac{2\pi r}{N})^3\frac{1}{(\frac{i2\pi r}{N}+s_{2})^{3+\gamma}}\Big(N^{2-2\delta}+\frac{N^{1-\delta}}{\frac{i2\pi r}{N}+s_{2}}\Big)\\
  &\leq \int_{1}^{\infty}CM(\frac{2\pi r}{N})^3\frac{1}{(\frac{x2\pi r}{N}+s_{2})^{3+\gamma}}\Big(N^{2-2\delta}+\frac{N^{1-\delta}}{\frac{x2\pi r}{N}+s_{2}}\Big)dx\\
  &\leq \frac{CM}{N^{2\delta}(\frac{2\pi r}{N}+s_{2})^{2+\gamma}}+\frac{CM}{N^{1+\delta}(\frac{2\pi r}{N}+s_{2})^{3+\gamma}},\\
\end{align*}

and multiplying by $s_{2}$ and integrating with respect to $s_{2}$ we obtain 

$$\int_{0}^{4N^{-1+\delta}}s_{2} (\frac{CM}{N^{2\delta}(\frac{2\pi r}{N}+s_{2})^{2+\gamma}}+\frac{CM}{N^{1+\delta}(\frac{2\pi r}{N}+s_{2})^{3+\gamma}})ds_{2}\leq CMN^{\gamma-\delta}.$$

Finally, we need to bound the integral when $(s_{1},s_{2})\in \mathcal{C}:=B_{4N^{-1+\delta}}(r,\alpha)-(r,\alpha)\setminus(\mathcal{A}\cup \mathcal{B})$. For this we only need to use that in this set $|s|\geq 3N^{-1+\delta}$ and that

$$|\int_{[-rS(r),rS(r)]\setminus [-\lfloor\frac{S(r)N}{2\pi}\rfloor \frac{2\pi r}{N},\lfloor\frac{S(r)N}{2\pi}\rfloor \frac{2\pi r}{N}]}ds_{1}|\leq 2\frac{2\pi r}{N}$$

so therefore

$$|\int_{\mathcal{C}}\frac{s_{2}\cos(\frac{N}{r}s_{1})(f_{N,\delta}^{pol}(r+s_{2},\frac{s_{1}}{r}+\alpha)-f_{N,\delta}^{pol}(r,\alpha))}{|s|^{3+\gamma}}ds_{1}ds_{2}|$$

$$\leq |\int_{-4N^{-1+\delta}}^{4N^{-1+\delta}}\frac{C |s_{2}|M }{N|N^{-1+\delta}|^{3+\gamma}}ds_{2}|\leq CMN^{\gamma-\delta-\delta\gamma}.$$




\end{proof}

\begin{lemma}\label{errorvsmall}
Given $\frac12>\delta>0$ and $1>\gamma>0,$, for any natural number $N$  fulfilling $N^{-\delta}\leq \frac{1}{100}$ and $N^{-1+\delta}<\frac{1}{100}$, a function $f_{N,\delta}(x)$ with $supp(f_{N,\delta})\subset B_{N^{-1+\delta}}(\cos(c_{1}),\sin(c_{1}))$ ($c_{1}\in\mathds{R}$), $||f_{N,\delta}^{pol}||_{C^j}\leq MN^{j(1-\delta)}$ for $j=0,1,2$  then we have that if $w_{N,\delta}^{pol}(r,\alpha):=f_{N,\delta}^{pol}(r,\alpha)\cos(N\alpha+c_{2})$ ($c_{2}\in\mathds{R}$) there exist constants $C,C_{\gamma}$ such that for $(r,\alpha)\in B_{2N^{-1+\delta}}(1,c_{1})$


$$|v^{pol}_{r,\gamma}(w_{N,\delta})(r,\alpha)-N^{\gamma} f_{N,\delta}^{pol}(r,\alpha)C_{\gamma} \sin (N\alpha+c_{2})|\leq CMN^{\gamma-\delta},$$

$$|v^{pol}_{\alpha,\gamma}(w_{N,\delta})(r,\alpha)|\leq CMN^{\gamma-\delta},$$
with $C_{\gamma}\neq 0$ depending on $\gamma$ and $C$ depending on $\gamma$ and $\delta$.

\end{lemma}

\begin{proof}

Using lemmas \ref{aproxkernel} and \ref{aproxfuncion} yields

\begin{align*}
    &|v^{pol}_{r,\gamma}(w_{N,\delta})(r,\alpha)-f_{N,\delta}^{pol}(r,\alpha)\int_{B^{pol}_{4N^{-1+\delta}}(r,\alpha)-(r,\alpha)}\frac{r^2\alpha'(\cos(N(\alpha'+\alpha)+c_{2})-\cos(N\alpha+c_{2}))}{|h^2+r^2(\alpha')^2|^{(3+\gamma)/2}}d\alpha'dh|\\
    &\leq CMN^{\gamma-\delta},\\
\end{align*}

\begin{align*}
    &|v^{pol}_{\alpha,\gamma}(w_{N,\delta})(r,\alpha)-f_{N,\delta}^{pol}(r,\alpha)\int_{B^{pol}_{4N^{-1+\delta}}(r,\alpha)-(r,\alpha)}\frac{rh(\cos(N(\alpha'+\alpha)+c_{2})-\cos(N\alpha+c_{2}))}{|h^2+r^2(\alpha')^2|^{(3+\gamma)/2}}d\alpha'dh|\\
    &\leq CMN^{\gamma-\delta},\\
\end{align*}
and therefore it is enough to prove

\begin{align}\label{vrlemmafinal}
    &|f_{N,\delta}^{pol}(r,\alpha)\int_{B^{pol}_{4N^{-1+\delta}}(r,\alpha)-(r,\alpha)}\frac{r^2\alpha'(\cos(N(\alpha'+\alpha)+c_{2})-\cos(N\alpha+c_{2}))}{|h^2+r^2(\alpha')^2|^{(3+\gamma)/2}}d\alpha'dh\\\nonumber
    &-N^{\gamma} C_{\gamma} \sin (N\alpha+c_{2})|\leq CMN^{\gamma-\delta},\\\nonumber
\end{align}

\begin{align}\label{valphalemmafinal}
    &|f_{N,\delta}^{pol}(r,\alpha)\int_{B^{pol}_{4N^{-1+\delta}}(r,\alpha)-(r,\alpha)}\frac{rh(\cos(N(\alpha'+\alpha)+c_{2})-\cos(N\alpha+c_{2}))}{|h^2+r^2(\alpha')^2|^{(3+\gamma)/2}}d\alpha'dh|\\\nonumber
    &\leq CMN^{\gamma-\delta}.\\\nonumber
\end{align}

We start with (\ref{valphalemmafinal}), where by using the odd symmetry of the integrand with respect to $h$

\begin{align*}
    &|f_{N,\delta}^{pol}(r,\alpha)\int_{B^{pol}_{4N^{-1+\delta}}(r,\alpha)-(r,\alpha)}\frac{rh(\cos(N(\alpha'+\alpha)+c_{2})-\cos(N\alpha+c_{2}))}{|h^2+r^2(\alpha')^2|^{(3+\gamma)/2}}d\alpha'dh|\\
    &=|f_{N,\delta}^{pol}(r,\alpha)\int_{-S_{\infty}}^{S_{\infty}}\int_{P_{-}(\alpha')}^{P_{+}(\alpha')}\frac{rh(\cos(N(\alpha'+\alpha)+c_{2})-\cos(N\alpha+c_{2}))}{|h^2+r^2(\alpha')^2|^{(3+\gamma)/2}}dhd\alpha'|\\
    &=|f_{N,\delta}^{pol}(r,\alpha)\int_{-S_{\infty}}^{S_{\infty}}\int_{P_{-}(\alpha')}^{-P_{+}(\alpha')}\frac{rh(\cos(N(\alpha'+\alpha)+c_{2})-\cos(N\alpha+c_{2}))}{|h^2+r^2(\alpha')^2|^{(3+\gamma)/2}}dhd\alpha'|\\
    &\leq |M\int_{-S_{\infty}}^{S_{\infty}}\frac{CN^{-2+2\delta}}{N^{(-1+\delta)(2+\gamma)}}d\alpha'|\leq CM N^{(-1+\delta)(1-\gamma)}\leq CMN^{\gamma-\delta}\\
\end{align*}


where we used that $|P_{+}(\alpha')+P_{-}(\alpha')|\leq C N^{-2+2\delta}$,  $|S_{\infty}|\leq arccos(1-16\frac{N^{-2+2\delta}}{r^2})\leq C N^{-1+\delta} $ and that, for $h\in [P_{-}(\alpha'),-P_{+}(\alpha')]$

$$\frac{1}{|h^2+r^2(\alpha')^2|^{(2+\gamma)/2}}\leq \frac{C}{N^{(-1+\delta)(2+\gamma)}}.$$

For (\ref{vrlemmafinal}) we use
\begin{align*}
    &\int_{B^{pol}_{4N^{-1+\delta}}(r,\alpha)-(r,\alpha)}\frac{r^2\alpha'(\cos(N(\alpha'+\alpha)+c_{2})-\cos(N\alpha+c_{2}))}{|h^2+r^2(\alpha')^2|^{(3+\gamma)/2}}d\alpha'dh\\
    &=-\sin(N\alpha+c_{2})\int_{-4N^{-1+\delta}}^{4N^{-1+\delta}}\int_{-rS(h_{2})}^{rS(h_{2})}\frac{h_{1}\sin(N\frac{h_{1}}{r})}{|h_{1}^2+h_{2}^2|^{(3+\gamma)/2}}dh_{1}dh_{2}\\
    &=-\sin(N\alpha+c_{2})\int_{\mathds{R}}\int_{\mathds{R}}\frac{h_{1}\sin(N\frac{h_{1}}{r})}{|h_{1}^2+h_{2}^2|^{(3+\gamma)/2}}dh_{1}dh_{2}\\
    &+4\sin(N\alpha+c_{2})\int_{0}^{\infty}\int_{r\tilde{S}(h_{2})}^{\infty}\frac{h_{1}\sin(N\frac{h_{1}}{r})}{|h_{1}^2+h_{2}^2|^{(3+\gamma)/2}}dh_{1}dh_{2}\\
\end{align*}

where we just take 
\begin{equation*}
    \tilde{S}(h)=
    \begin{cases}
      S(h), & \text{if}\ h\in[-4N^{-1+\delta},4N^{-1+\delta}] \\
      0 & \text{otherwise.}
    \end{cases}
  \end{equation*}

But, we have that, for $i$ a natural number,

$$|\int_{i\frac{2\pi r}{N}+r\tilde{S}(h_{2})}^{(i+1)\frac{2\pi r}{N}+r\tilde{S}(h_{2})}\frac{h_{1}\sin(N\frac{h_{1}}{r})}{|h_{1}^2+h_{2}^2|^{(3+\gamma)/2}}dh_{1}|$$

$$\leq \frac{C}{N^2}\frac{1}{|(i\frac{2\pi r}{N}+r\tilde{S}(h_{2}))^2+h_{2}^2|^{(3+\gamma)/2}}$$
and thus 

$$|\int_{r\tilde{S}(h_{2})}^{\infty}\frac{h_{1}\sin(N\frac{h_{1}}{r})}{|h_{1}^2+h_{2}^2|^{(3+\gamma)/2}}dh_{1}|$$

$$\leq \sum_{i=0}^{\infty}\frac{C}{N^2}\frac{1}{|(i\frac{2\pi r}{N}+r\tilde{S}(h_{2}))^2+h_{2}^2|^{(3+\gamma)/2}}$$

$$\leq \int_{-1}^{\infty}\frac{C}{N^2}\frac{1}{|x\frac{2\pi r}{N}+r\tilde{S}(h_{2})+h_{2}|^{(3+\gamma)}}dx$$
$$\leq \frac{C}{N|-\frac{2\pi r}{N}+r\tilde{S}(h_{2})+h_{2}|^{(2+\gamma)}}\leq \frac{C}{N|r\tilde{S}(h_{2})+h_{2}|^{(2+\gamma)}}$$
where we used for $h_{2}>0$, $r\geq\frac12$ we have $r\tilde{S}(h_{2})+h_{2}\geq CN^{-1+\delta}$. But then

$$|4\sin(N\alpha+c_{2})\int_{0}^{N^{-1+\delta}}\int_{r\tilde{S}(h_{2})}^{\infty}\frac{h_{1}\sin(N\frac{h_{1}}{r})}{|h_{1}^2+h_{2}^2|^{(3+\gamma)/2}}dh_{1}dh_{2}|$$
$$\leq  \int_{0}^{N^{-1+\delta}}\frac{C}{N^{1+(-1+\delta)(2+\gamma)}} dh_{2}=CN^{\gamma-\delta-\delta\gamma}$$
and

$$|4\sin(N\alpha+c_{2})\int_{N^{-1+\delta}}^{\infty}\int_{r\tilde{S}(h_{2})}^{\infty}\frac{h_{1}\sin(N\frac{h_{1}}{r})}{|h_{1}^2+h_{2}^2|^{(3+\gamma)/2}}dh_{1}dh_{2}|$$

$$\leq |\int_{N^{-1+\delta}}^{\infty}\frac{C}{N|h_{2}|^{(2+\gamma)}}dh_{2}|\leq CN^{\gamma-\delta-\gamma\delta},$$
and therefore

\begin{align*}
    &|\int_{B^{pol}_{4N^{-1+\delta}}(r,\alpha)-(r,\alpha)}\frac{r^2\alpha'(\cos(N(\alpha'+\alpha)+c_{2})-\cos(N\alpha+c_{2}))}{|h^2+r^2(\alpha')^2|^{(3+\gamma)/2}}d\alpha'dh\\
    &+\sin(N\alpha+c_{2})\int_{\mathds{R}}\int_{\mathds{R}}\frac{h_{1}\sin(N\frac{h_{1}}{r})}{|h_{1}^2+h_{2}^2|^{(3+\gamma)/2}}dh_{1}dh_{2}|\leq CN^{\gamma-\delta-\gamma\delta}\\
\end{align*}
and combined with (\ref{vrlemmafinal}) we get

\begin{align*}
    &|v^{pol}_{\alpha,\gamma}(w_{N,\delta})(r,\alpha)+f_{N,\delta}^{pol}(r,\alpha)\sin(N\alpha+c_{2})\int_{\mathds{R}}\int_{\mathds{R}}\frac{h_{1}\sin(N\frac{h_{1}}{r})}{|h_{1}^2+h_{2}^2|^{(3+\gamma)/2}}dh_{1}dh_{2}|\\
    &\leq CMN^{\gamma-\delta}.\\
\end{align*}


Furthermore

\begin{align*}
    &-\sin(N\alpha+c_{2})\int_{\mathds{R}^2}\frac{h_{1}\sin(\frac{N}{r}h_{1})}{|h^{2}_{1}+h^{2}_{2}|^{\frac{3+\gamma}{2}}}dh_{1}dh_{2}=-\sin(N\alpha+c_{2})\bigg(\frac{N}{r}\bigg)^{\gamma}\int_{\mathds{R}^2}\frac{h_{1}\sin(h_{1})}{|h_{1}^{2}+h_{2}^{2}|^{\frac{3+\gamma}{2}}}dh_{1}dh_{2},\\
\end{align*}
and

\begin{align*}
    &\int_{\mathds{R}^2}\frac{h_{1}\sin(h_{1})}{|h_{1}^{2}+h_{2}^{2}|^{\frac{3+\gamma}{2}}}dh_{1}dh_{2}=\int_{-\infty}^{\infty}h_{1}\sin(h_{1})\int_{-\infty}^{\infty}\frac{1}{(h_{1}^{2}+h_{2}^{2})^{\frac{(3+\gamma)}{2}}}dh_{2}dh_{1}\\
    &=\int_{-\infty}^{\infty}\frac{h_{1}\sin(h_{1})}{|h_{1}|^{2+\gamma}}\int_{-\infty}^{\infty}\frac{1}{(1+\lambda^{2})^{\frac{(3+\gamma)}{2}}}d\lambda dh_{1}=K_{\gamma}2\int_{0}^{\infty}\frac{h_{1}\sin(h_{1})}{|h_{1}|^{2+\gamma}}dh_{1}.\\\
\end{align*}
By using that $\frac{h_{1}}{|h_{1}|^{2+\gamma}}$ is monotone decreasing for $h_{1}>0$, $\sin(x+\pi)=-\sin(x)$, $\sin(x)>0$ if $x\in(0,\pi)$  and $K_{\gamma}>0$ we obtain

$$C_{\gamma}:=-K_{\gamma}2\int_{0}^{\infty}\frac{h_{1}\sin(h_{1})}{|h_{1}|^{2+\gamma}}<0.$$
Thus

$$|v^{pol}_{r,\gamma}(w_{N,\delta})(r,\alpha)-f_{N,\delta}^{pol}(r,\alpha)\bigg(\frac{N}{r}\bigg)^{\gamma} C_{\gamma} \sin (N\alpha+c_{2})|\leq CMN^{\gamma-\delta},$$
and since, for the values of $r$ considered we have 
$$|\bigg(\frac{N}{r}\bigg)^{\gamma}-N^{\gamma}|\leq CN^{\gamma-1+\delta}\leq CN^{\gamma-\delta}$$
we are done.

\end{proof}



\begin{lemma}\label{decaimiento}


Given $0<\delta<\frac{1}{2}$, $0<\gamma<1$, a natural number $N$ such that $N^{-1+\delta}\leq \frac{1}{100}$ and a $C^2$ function $f_{N,\delta}$,  satisfying $supp(f_{N,\delta})\subset B_{N^{-1+\delta}}(\cos(c_{1}),\sin(c_{1}))$ ($c_{1}\in\mathds{R}$) with $||f_{N,\delta}||_{C^{j}}\leq M N^{j(1-\delta)}$, $j=0,1,2$, then for any $x=(x_{1},x_{2})=(R\cos(A),R\sin(A))\in \mathds{R}^2\setminus B_{2N^{-1+\delta}}(\cos(c_{1}),\sin(c_{1}))$ we have that

$$|v^{pol}_{r,\gamma}(f_{N,\delta}(r,\alpha)\sin(N\alpha))(R,A)|\leq C\frac{M}{|d(x,f_{N,\delta})|^{2+\gamma}} N^{-2+\delta},$$

$$|v^{pol}_{\alpha,\gamma}(f_{N,\delta}(r,\alpha)\sin(N\alpha))(R,A)|\leq C\frac{M}{|d(x,f_{N,\delta})|^{2+\gamma}} N^{-2+\delta}$$
with $C$ depending only on $\gamma$.

Furthermore, if $f_{N,\delta}\in C^{k+2}$ for $k$ an integer $k\geq 1$ and $||f_{N,\delta}(r,\alpha)||_{C^{j}}\leq M N^{j(1-\delta)}$ for $j=0,1,...,k$ then we have 

$$|\frac{\partial^{j}v^{pol}_{r,\gamma}(f_{N,\delta}(r,\alpha)\sin(N\alpha))(R,A)}{\partial x_{1}^{l}\partial x_{2}^{j-l}}|\leq C\frac{M}{|d(x,f_{N,\delta})|^{2+\gamma}} N^{-2+\delta+j},$$

$$|\frac{\partial^{j}v^{pol}_{\alpha,\gamma}(f_{N,\delta}(r,\alpha)\sin(N\alpha))(R,A)}{\partial x_{1}^{l}\partial x_{2}^{j-l}}|\leq C\frac{M}{|d(x,f_{N,\delta})|^{2+\gamma}} N^{-2+\delta+j}$$
for $j=0,1,...,k+2$, $l=0,1,...,j$, with $C$ depending on $\gamma$ and $j$.
\end{lemma}

\begin{proof}
We will consider $c_{1}=0$ for simplicity and we will obtain the expression only for $v_{r,\gamma}$, $v_{\alpha,\gamma}$ being equivalent. That is to say, we want to compute
\begin{align*}
   &\int_{supp(f^{pol}_{N,\delta})}\frac{(r')^2\sin(\alpha'-A)}{|(R-r')^2+2Rr'(1-\cos(A-\alpha'))|^{(3+\gamma)/2}}f_{N,\delta}(r',\alpha')\sin(N\alpha')d\alpha'dr'\\
   &=\cos(NA)\int_{supp(f^{pol}_{N,\delta})}\frac{(r')^2\sin(\alpha'-A)f_{N,\delta}(r',\alpha')\sin(N\alpha'-NA)}{|(R-r')^2+2Rr'(1-\cos(A-\alpha'))|^{(3+\gamma)/2}}d\alpha'dr'\\
   &+\sin(NA)\int_{supp(f^{pol}_{N,\delta})}\frac{(r')^2\sin(\alpha'-A)f_{N,\delta}(r',\alpha')\cos(N\alpha'-NA)}{|(R-r')^2+2Rr'(1-\cos(A-\alpha'))|^{(3+\gamma)/2}}d\alpha'dr'\\
   &=\cos(NA)\int_{supp(f^{pol}_{N,\delta})-(0,A)}\frac{(r')^2\sin(\bar{\alpha})f_{N,\delta}(r',\bar{\alpha}+A)\sin(N\bar{\alpha})}{|(R-r')^2+2Rr'(1-\cos(\bar{\alpha}))|^{(3+\gamma)/2}}d\bar{\alpha}dr'\\
   &+\sin(NA)\int_{supp(f^{pol}_{N,\delta})-(0,A)}\frac{(r')^2\sin(\bar{\alpha})f_{N,\delta}(r',\bar{\alpha}+A)\cos(N\bar{\alpha})}{|(R-r')^2+2Rr'(1-\cos(\bar{\alpha}))|^{(3+\gamma)/2}}d\bar{\alpha}dr'\\
\end{align*}
with $f$, $R$ and $A$ as in the hypothesis of the lemma. We will focus on the part depending on $cos(NA)$, the other term being analogous. First, a second order Taylor expansion and some computations give us, since $r'\in(\frac12,\frac32)$

\begin{align*}
   &|\int_{i\frac{2\pi}{N}+\frac{\pi}{2N}}^{(i+1)\frac{2\pi}{N}+\frac{\pi}{2N}}\frac{(r')^2\sin(\bar{\alpha})}{|(R-r')^2+2Rr'(1-\cos(\bar{\alpha}))|^{(3+\gamma)/2}}f_{N,\delta}(r',\bar{\alpha}+A)\sin(N\bar{\alpha})d\bar{\alpha}|\\ 
   &\leq \int_{i\frac{2\pi}{N}+\frac{\pi}{2N}}^{(i+1)\frac{2\pi}{N}+\frac{\pi}{2N}}\bigg(\frac{2\pi}{N}\bigg)^2 sup_{\bar{\alpha}\in[i\frac{2\pi}{N}+\frac{\pi}{2N},(i+1)\frac{2\pi}{N}+\frac{\pi}{2N}]}\Big(\Big|\frac{\partial^2}{\partial \bar{\alpha}^2}\frac{(r')^2\sin(\bar{\alpha})f_{N,\delta}(r',\bar{\alpha}+A)}{|(R-r')^2+2Rr'(1-\cos(\bar{\alpha}))|^{(3+\gamma)/2}}\Big|\Big)\\
   &\times|\sin(N\bar{\alpha})|d\bar{\alpha}\\
   &\leq C\bigg(\frac{2\pi}{N}\bigg)^3\Big(\frac{||f_{N,\delta}(r,\alpha)||_{C^{2}}}{|(R-r')^2+2Rr'(1-\cos(\bar{\alpha}))|^{(2+\gamma)/2}}+\frac{||f_{N,\delta}(r,\alpha)||_{C^{1}}}{|(R-r')^2+2Rr'(1-\cos(\bar{\alpha}))|^{(3+\gamma)/2}}\\
   &+\frac{||f_{N,\delta}(r,\alpha)||_{L^{\infty}}}{|(R-r')^2+2Rr'(1-\cos(\bar{\alpha}))|^{(4+\gamma)/2}}\Big).\\
\end{align*}
Using that, for $(r',\bar{\alpha})\in supp(f^{pol}_{N,\delta})-(0,A)$ 
$$|(R-r')^2+2Rr'(1-\cos(\bar{\alpha}))|^{\frac12}\geq d((R,A),f_{N,\delta})$$

$$d((R,A),f_{N,\delta})\geq N^{-1+\delta}$$
and the properties of $f_{N,\delta}$ we get then that

\begin{align*}
   &|\int_{i\frac{2\pi}{N}+\frac{\pi}{2N}}^{(i+1)\frac{2\pi}{N}+\frac{\pi}{2N}}\frac{(r')^2\sin(\bar{\alpha})}{|(R-r')^2+2Rr'(1-\cos(\bar{\alpha}))|^{(3+\gamma)/2}}f_{N,\delta}(r',\bar{\alpha}+A)\sin(N\bar{\alpha})d\bar{\alpha}|\\ 
   &\leq \frac{CM N^{-1-2\delta}}{d((R,A),f_{N,\delta})^{2+\gamma}}\\
\end{align*}
so that

\begin{align*}
    &\Big|\int_{1-N^{-1+\delta}}^{1+N^{-1+\delta}}\int_{-\lfloor\frac{S(r')N}{2\pi}\rfloor \frac{2\pi}{N}+\frac{\pi}{2N}}^{\lfloor\frac{S(r')N}{2\pi}\rfloor \frac{2\pi}{N}-\frac{3\pi}{2N}} \frac{(r')^2\sin(\bar{\alpha})}{|(R-r')^2+2Rr'(1-\cos(\bar{\alpha}))|^{(3+\gamma)/2}}f_{N,\delta}(r',\bar{\alpha}+A)\sin(N\bar{\alpha})d\bar{\alpha}dr'\Big|\\
    &\leq \int_{1-N^{-1+\delta}}^{1+N^{-1+\delta}}\frac{CM N^{-1-\delta}}{d((R,A),f_{N,\delta})^{2+\gamma}}dr'\leq \frac{CM }{N^2 d((R,A),f_{N,\delta})^{2+\gamma}}.\\
\end{align*}

As for the rest of the integral we have

\begin{align*}
    &\Big|\int_{1-N^{-1+\delta}}^{1+N^{-1+\delta}}\int_{\lfloor\frac{S(r')N}{2\pi}\rfloor \frac{2\pi}{N}-\frac{3\pi}{2N}}^{S(r')} \frac{(r')^2\sin(\bar{\alpha})}{|(R-r')^2+2Rr'(1-\cos(\bar{\alpha}))|^{(3+\gamma)/2}}f_{N,\delta}(r',\bar{\alpha}+A)\sin(N\bar{\alpha})d\bar{\alpha}dr'\Big|\\
    &\leq \int_{1-N^{-1+\delta}}^{1+N^{-1+\delta}}\frac{CM N^{-1}}{d((R,A),f_{N,\delta})^{2+\gamma}}dr'\leq \frac{CM N^{\delta} }{N^2 d((R,A),f_{N,\delta})^{2+\gamma}},\\
\end{align*}
and
\begin{align*}
    &\Big|\int_{1-N^{-1+\delta}}^{1+N^{-1+\delta}}\int_{-S(r')}^{-\lfloor\frac{S(r')N}{2\pi}\rfloor \frac{2\pi}{N}+\frac{\pi}{2N}} \frac{(r')^2\sin(\bar{\alpha})}{|(R-r')^2+2Rr'(1-\cos(\bar{\alpha}))|^{(3+\gamma)/2}}f_{N,\delta}(r',\bar{\alpha}+A)\sin(N\bar{\alpha})d\bar{\alpha}dr'\Big|\\
    &\leq \int_{1-N^{-1+\delta}}^{1+N^{-1+\delta}}\frac{CM N^{-1}}{d((R,A),f_{N,\delta})^{2+\gamma}}dr'\leq \frac{CM N^{\delta} }{N^2 d((R,A),f_{N,\delta})^{2+\gamma}}\\
\end{align*}
and we are done.

To obtain the result for the derivatives, we first note that since
\begin{align}\label{v1v2rad}
 &v_{1,\gamma}(w)=\cos(\alpha)v_{r,\gamma}(w)-\sin(\alpha)v_{\alpha,\gamma}(w)\\\nonumber
 &v_{2,\gamma}(w)=\sin(\alpha)v_{r,\gamma}(w)+\cos(\alpha)v_{\alpha,\gamma}(w)\\\nonumber
\end{align}
then for $x=(x_{1},x_{2})=(R\cos(A),R\sin(A))$

$$|v^{pol}_{1,\gamma}(f_{N,\delta}(r,\alpha)\sin(N\alpha))(R,A)|\leq C\frac{M}{|d(x,f_{N,\delta})|^{2+\gamma}} N^{-2+\delta},$$

$$|v^{pol}_{2,\gamma}(f_{N,\delta}(r,\alpha)\sin(N\alpha))(R,A)|\leq C\frac{M}{|d(x,f_{N,\delta})|^{2+\gamma}} N^{-2+\delta}.$$

Furthermore, derivation commutes with the operators $v_{1,\gamma}$ and $v_{2,\gamma}$, so we can prove that

$$|\frac{\partial^{j}v^{pol}_{1,\gamma}(f_{N,\delta}(r,\alpha)\sin(N\alpha))(R,A)}{\partial x_{1}^{l}\partial x_{2}^{j-l}}|\leq C\frac{M}{|d(x,f_{N,\delta})|^{2+\gamma}} N^{-2+\delta+j},$$

$$|\frac{\partial^{j}v^{pol}_{2,\gamma}(f_{N,\delta}(r,\alpha)\sin(N\alpha))(R,A)}{\partial x_{1}^{l}\partial x_{2}^{j-l}}|\leq C\frac{M}{|d(x,f_{N,\delta})|^{2+\gamma}} N^{-2+\delta+j}$$ by differentiating $f_{N,\delta}(r,\alpha)\sin(N\alpha)$ and applying our lemma for each individual term.

Then, using (\ref{vrv1v2}) and computing $\frac{\partial^{j}}{\partial x_{1}^{l}\partial x_{2}^{j-l}} v_{r,\gamma}$, $\frac{\partial^{j}}{\partial x_{1}^{l}\partial x_{2}^{j-l}} v_{\alpha,\gamma}$ we obtain, for $r\geq\frac12$ that

\begin{align*}
    &\frac{\partial^{j}}{\partial x_{1}^{l}\partial x_{2}^{j-l}} v_{r,\gamma}(R,A)\leq  C\Big( \sum_{k=0}^{j}\sum_{l=0}^{k} |\frac{\partial^{k} v_{1,\gamma}}{\partial x_{1}^{l}\partial x_{2}^{k-l}} (R,A)|+|\frac{\partial^{k} v_{2,\gamma}}{\partial x_{1}^{l}\partial x_{2}^{k-l}} (R,A)|\Big)\\
    &\leq C\frac{M}{|d(x,f_{N,\delta})|^{2+\gamma}} N^{-2+\delta+j},\\
\end{align*}

\begin{align*}
    &\frac{\partial^{j}}{\partial x_{1}^{l}\partial x_{2}^{j-l}} v_{\alpha,\gamma}(R,A)\leq  C\Big( \sum_{k=0}^{j}\sum_{l=0}^{k} |\frac{\partial^{k} v_{1,\gamma}}{\partial x_{1}^{l}\partial x_{2}^{k-l}} (R,A)|+|\frac{\partial^{k} v_{2,\gamma}}{\partial x_{1}^{l}\partial x_{2}^{k-l}} (R,A)|\Big)\\
    &\leq C\frac{M}{|d(x,f_{N,\delta})|^{2+\gamma}} N^{-2+\delta+j},\\
\end{align*}

and we are done.

\end{proof}

\section{Pseudo-solutions considered and their properties}

To obtain ill-posedness for the space $C^{k,\beta}$ for $\gamma$-SQG, we will add perturbations to a radial solution $f(r)$  (with $f(r)$ chosen so that it has some specific properties). These perturbations will be of the form

\begin{equation}\label{pertbase}
    \lambda\sum_{l=0}^{L-1} f(N^{1-\delta}(r-1),N^{1-\delta}\alpha)\frac{\cos(N(M+l)(\alpha-\alpha^{1})+\alpha^{2}+\frac{k\pi}{2})}{L(NM)^{k+\beta}},
\end{equation}
with

\begin{itemize}
    \item $f(r-1,\alpha)=g(r-1)g(\alpha)$, $g$ a positive $C^\infty$ function with support in $[-\frac12,\frac12]$ and such that $f(x)=1$ if $ x\in[-\frac14,\frac14]$ and $||f(r-1,\alpha)||_{C^{j}}\leq 100^{j}$,
    \item $M,N,\lambda>0$,  $\delta\in(0,\frac{1}{2})$, $L\in\mathds{N}$ and $\alpha^{1},\alpha^{2}\in \mathds{R}$,
    \item $N^{\delta}\geq 100$, $N^{1-\delta}\geq 100$,
    \item $k\in\mathds{N}$, $\beta\in(0,1]$, $\gamma\in(0,1)$,
    \item $k+\beta>1+2\delta+\gamma$,
    \item $L<\frac{M}{2}$.
    
\end{itemize}

For compactness of notation, whenever we have $f,\delta,N,L,M,\lambda$ satisfying these properties we will say that they satisfy the usual conditions. From now on we will consider $k$, $\beta$, $\gamma$ and $\delta$  fixed satisfying these properties, just so that we can avoid extra sub-indexes for these parameters. Due to this, one needs to keep in mind that in general the constants in the lemmas obtained might depend on the specific values of $k$, $\beta$, $\gamma$ and $\delta$.
Before we study how this kind of perturbations will evolve with time, we start by obtaining some basic properties regarding the norms of (\ref{pertbase}).

\begin{lemma}\label{normasckpert}
Given a perturbation as in (\ref{pertbase}), which we will refer as $w_{k,\beta}$, with $f,\delta,N,L,M,\lambda$ satisfying the usual conditions we have that

$$||w_{k,\beta}||_{C^{j}}\leq C K_{j}\lambda (NM)^{j-k-\beta}$$

$$|\frac{\partial^{k}w_{k,\beta}(r,\alpha)}{\partial^{k-i} x_{1}\partial^{i}x_{2}}|\leq \frac{C\lambda}{L |\sin(N\frac{\alpha-\alpha^{1}}{2})|(NM)^{\beta}}+C\lambda (NM)^{-\delta-\beta}+\frac{C\lambda L}{M(NM)^{\beta}}$$

$$|\frac{\partial^{k+1}w_{k,\beta}(r,\alpha)}{\partial^{k+1-i} x_{1}\partial^{i}x_{2}}|\leq \frac{C\lambda(NM)^{1-\beta}}{L |\sin(N\frac{\alpha-\alpha^{1}}{2})|}+C\lambda (NM)^{-\delta-\beta+1}+\frac{C\lambda(NM)^{1-\beta}L}{M}$$

with $C$ a constant depending on $f$ and $K_{j}$ constants depending on $j$.
\end{lemma}

\begin{proof}
The bounds for the $C^{j}$ norms can be obtained directly by using that, for functions with support concentrated around $r=1$, we have that

$$||f(x_{1},x_{2})||_{C^{j}}\leq K_{j}||f^{pol}(r,\alpha)||_{C^{j}}$$
and the bounds for the derivatives of $w^{pol}_{k,\beta}$ can be obtained by direct computation. For the other two inequalities, we have that
\begin{align*}
    &|\frac{\partial^{k}w_{k,\beta}(r,\alpha)}{\partial^{k-i} x_{1}\partial^{i}x_{2}}|\leq K_{k}||w_{k,\beta}^{pol}(r,\alpha)||_{C^{k}}\leq C\lambda (NM)^{-\beta-\delta}\\
    &+C\lambda|\sum_{l=0}^{L-1}f(N^{1-\delta}(r-1),N^{1-\delta}\alpha)\frac{\partial^k}{\partial \alpha^{k}}(\frac{\cos(N(M+l)(\alpha-\alpha^{1})+\alpha^{2}+\frac{k\pi}{2})}{L(NM)^{k+\beta}})|\\
    &\leq C\lambda (NM)^{-\beta-\delta}+\frac{C\lambda L}{M(NM)^{\beta}}\\
    &+C\lambda|\sum_{l=0}^{L-1}f(N^{1-\delta}(r-1),N^{1-\delta}\alpha)\frac{\cos(N(M+l)(\alpha-\alpha^{1})+\alpha^{2})}{L(NM)^{\beta}}|\\
\end{align*}
but we can compute  $\sum_{l=0}^{L-1}\cos(N(M+l)(\alpha-\alpha^{1})+\alpha^{2})$ as

$$\sum_{l=0}^{L-1}\cos(N(M+l)(\alpha-\alpha^{1})+\alpha^{2})$$
$$=\frac{\sin(\frac{NL(\alpha-\alpha^{1})}{2})}{\sin(N\frac{\alpha-\alpha^{1}}{2})}\cos(NM(\alpha-\alpha^{1})+\alpha^{2}+\frac{N(L-1)(\alpha-\alpha^{1})}{2})$$
which gives us

\begin{align*}
    &|\frac{\partial^{k}w_{k,\beta}(r,\alpha)}{\partial^{k-i} x_{1}\partial^{i}x_{2}}|\\
    &\leq C\lambda (NM)^{-\beta-\delta}+\frac{C\lambda L}{M(NM)^{\beta}}+\frac{C\lambda }{|\sin(N\frac{\alpha-\alpha^{1}}{2})|L(NM)^{\beta}}.\\
\end{align*}

The proof with $k+1$ derivatives is done analogously.
\end{proof}

This lemma tells us that these perturbations behave similarly to wave packets, with their amplitude and derivatives decreasing as one gets further from $\alpha^{1}+j\frac{2\pi}{N}$ . We will use this property to obtain upper bounds for the norms of these perturbations when several of them are placed appropriately far way from each other. For this, we first we need a short technical lemma.

\begin{lemma}\label{interpolacioncalfa}
Given a $C^1$ function $f(x):\mathds{R}\rightarrow \mathds{R}$ with $||f(x)||_{L^{\infty}}\leq M_{1}$ and $||f'(x)||_{L^{\infty}}\leq M_{2}$, we have that, for any $x,h\in \mathds{R}$, $\beta\in(0,1)$

$$\frac{|f(x)-f(x+h)|}{|h|^{\beta}}\leq 2^{1-\beta}M_{1}^{1-\beta}M_{2}^{\beta}.$$
\end{lemma}
\begin{proof}
We have the two trivial bounds

\begin{equation*}
    \frac{|f(x)-f(x+h)|}{|h|^{\beta}}\leq \frac{2M_{1}}{|h|^{\beta}}
\end{equation*}

\begin{equation*}
    \frac{|f(x)-f(x+h)|}{|h|^{\beta}}\leq \frac{|h|M_{2}}{|h|^{\beta}},
\end{equation*}

and thus it is enough to find a bound for

$$sup_{h\in \mathds{R}}(min(\frac{2M_{1}}{|h|^{\beta}},\frac{|h|M_{2}}{|h|^{\beta}})).$$

But it is easy to see that the supremum is attained when $\frac{2M_{1}}{|h|^{\beta}}=\frac{|h|M_{2}}{|h|^{\beta}}$. Since this happens when $|h|=\frac{2M_{1}}{M_{2}}$, substituting $|h|$ in any of the upper bounds gives us

$$\frac{|f(x)-f(x+h)|}{|h|^{\beta}}\leq \frac{2M_{1}}{\big(\frac{2M_{1}}{M_{2}}\big)^{\beta}}=(2M_{1})^{1-\beta}M_{2}^{\beta}.$$
\end{proof}
Now we are ready to prove decay in space of the functions that we use as perturbations.

\begin{lemma}\label{normackbetapert}
Given a function $g(x)$ of the form

$$g^{pol}(r,\alpha)=\sum_{j=1}^{J}\lambda_{j}\sum_{l=0}^{L-1} f(N^{1-\delta}(r-1),N^{1-\delta}\alpha)\frac{\cos(N(M_{j}+l)(\alpha-\alpha_{j}^{1})+\alpha_{j}^{2}+\frac{k\pi}{2})}{JL(NM_{j})^{k+\beta}}$$
where $f,\delta,N,L,M_{j},\lambda_{j}$ satisfy the usual conditions and
with $\alpha^{1}_{j}\in[c\frac{\pi}{N},\frac{\pi}{N}]$ and $|\alpha^{1}_{j_{1}}-\alpha^{1}_{j_{2}}|\geq c\frac{\pi}{N}$ for some $c>0$ and $\frac{M_{j_{1}}}{M_{j_{2}}}\leq 2$ for $j_{1},j_{2}\in\{1,2,...,J\}$ then we have that

$$|g|_{C^{k,\beta}}\leq C\bar{\lambda}(\frac{1}{J}+\frac{1}{(N\bar{M})^{\delta}}+\frac{L}{\bar{M}}+\frac{1}{c L})$$
with $C$ depending on $k$, $\beta$ and $\delta$ and where $\bar{M}:=sup_{j=1,...,J}(M_{j})$, $\bar{\lambda}:=sup_{j=1,...,J}(\lambda_{j}).$

\end{lemma}

\begin{proof}

We will compute bounds for the seminorm $|\cdot|_{C^{\alpha}}$ of an arbitrary k-th derivative of $g$, and we will refer to it simply as $g^{(k)}(x)$ since the specific derivative we consider is irrelevant for the proof and we will use $d^{k}$ as notation for the specific k-th derivative for the same reason. We start by obtaining bounds for $||g^{(k)}||_{L^{\infty}}$. Since $|\alpha^{1}_{j_{1}}-\alpha^{1}_{j_{2}}|\geq c\frac{\pi}{N}$ and $\alpha^{1}_{j}\in [c\frac{\pi}{N},\frac{\pi}{N}]$, we have that for any $\alpha$ there is at most one $j$ with

\begin{equation}\label{jcercano}
    min_{n\in \mathds{Z}}|\alpha-\alpha^{1}_{j}-\frac{\pi n}{N}|< \frac{c\pi}{2N}.
\end{equation}

For simplicity, assume that $j=1$ fulfils (\ref{jcercano}) (the proof when other values of $j$ or no value of $j$ fulfil (\ref{jcercano}) is equivalent).

Then, using lemma \ref{normasckpert} we obtain

\begin{align*}
    &|g^{(k)}(r,\alpha)|\leq \\
    &\leq C|\lambda_{1}\sum_{l=0}^{L-1}d^{k}\big( f(N^{1-\delta}(r-1),N^{1-\delta}\alpha)\frac{\cos(N(M_{1}+l)(\alpha-\alpha_{j}^{1})+\alpha_{1}^{2}+\frac{k\pi}{2})}{JL(NM_{1})^{k+\beta}}\big)|\\
    &+C|\sum_{j=2}^{J}\lambda_{j}\sum_{l=0}^{L-1} d^{k}\big(f(N^{1-\delta}(r-1),N^{1-\delta}\alpha)\frac{\cos(N(M_{j}+l)(\alpha-\alpha_{j}^{1})+\alpha_{j}^{2}+\frac{k\pi}{2})}{JL(NM_{j})^{k+\beta}}\big)|\\
    &\leq \frac{C\bar{\lambda}}{J(N\bar{M})^{\beta}}+\frac{C\bar{\lambda}}{(N\bar{M})^{\beta+\delta}}+\frac{C\bar{\lambda} L}{\bar{M}(N\bar{M})^{\beta}}+\frac{C\bar{\lambda} }{|\sin(\frac{c\pi}{2})|L(N\bar{M})^{\beta}}\\
    &\leq \frac{C\bar{\lambda}}{J(N\bar{M})^{\beta}}+\frac{C\bar{\lambda}}{(N\bar{M})^{\beta+\delta}}+\frac{C\bar{\lambda} L}{\bar{M}(N\bar{M})^{\beta}}+\frac{C\bar{\lambda} }{c L(N\bar{M})^{\beta}}\\
    &=\frac{C\bar{\lambda}}{(\bar{M}N)^{\beta}}(\frac{1}{J}+\frac{1}{(N\bar{M})^{\delta}}+\frac{L}{\bar{M}}+\frac{1}{c L}).\\
\end{align*}

Arguing the same way for any arbitrary $k+1$ derivative we obtain

\begin{align*}
    &|g^{(k+1)}(r,\alpha)| \\
    &\leq N\bar{M} \frac{C\bar{\lambda}}{(\bar{M}N)^{\beta}}(\frac{1}{J}+\frac{1}{(N\bar{M})^{\delta}}+\frac{L}{\bar{M}}+\frac{1}{c L}).\\
\end{align*}
and then direct application of lemma \ref{interpolacioncalfa} gives us

$$\frac{g^{(k)}(x)-g^{(k)}(x+h)}{|h|^{\beta}}\leq C\bar{\lambda}(\frac{1}{J}+\frac{1}{(N\bar{M})^{\delta}}+\frac{L}{\bar{M}}+\frac{1}{c L}).$$
\end{proof}

With this out of the way, we are ready to define the pseudo-solutions that we will use to prove ill-posedness. Namely, 
we define

\begin{align}
    &\bar{w}^{pol}_{\lambda,N,M,J,L,\tilde{t}}(r,\alpha,t):= \lambda_{0} f_{1}(r)+ \sum_{j=1}^{J}\sum_{l=0}^{L-1}\bigg(\lambda_{j}f_{2}(N^{1-\delta}(r-1),N^{1-\delta}(\alpha-t\lambda_{0} v_{\alpha,\gamma}(f_{1})(r=1)))\\ \nonumber
    & \frac{\cos(N(M_{j}+l)(\alpha-\alpha_{j}^{1}-t\lambda_{0} v_{\alpha,\gamma}(f_{1})(r=1))+\alpha_{j}^{2}+\frac{k\pi}{2}+t\lambda_{0} C_{\gamma}N^{\gamma}(M_{j}+l)^{\gamma}}{JL(NM_{j})^{k+\beta}}\big)\\ \nonumber
\end{align}
with

$$M_{j}=M(1+\frac{j}{J}),\ \lambda_{0}=\frac{\pi M^{1-\gamma}}{2\tilde{t}N^{\gamma}C_{\gamma}\gamma},\ \lambda_{j}=\lambda (1+\frac jJ)^{\beta}\text{ for } j=1,...,J,$$ 
$$\ \alpha^{1}_{j}=\frac{\pi}{2N}(1+\frac{j}{J})^{-1+\gamma}-\tilde{t}\lambda_{0} v_{\alpha,\gamma}(f_{1})(r=1),\ \alpha^{2}_{j}=-(\frac{1}{\gamma}-1)\frac{\pi}{2} M(1+\frac{j}{J})^{\gamma}.$$
The functions $f_{1,\gamma}(r)$ and $f_{2}(r-1,\alpha)$ and the values $k,\beta,\gamma,\delta,\lambda,N,M,J,L$ and $\tilde{t}$ will fulfil the following properties:

\begin{itemize}
    \item  $\lambda,N,M,J,L,\tilde{t}>0$,  $\delta\in(0,\frac12),\gamma\in(0,1)$ and  $L,J,M\in\mathds{N}$, $\frac{M}{J}\in\mathds{N}$,
        \item $f_{2}(r-1,\alpha)=g(r-1)g(\alpha)$, $g$ a positive $C^\infty$ function with support in $[-\frac12,\frac12]$ and such that $f(x)=1$ if $ x\in[-\frac14,\frac14]$ and $||f_{2}(r-1,\alpha)||_{C^{j}}\leq 100^{j}$,  
    \item $N^{\delta}\geq 100$, $N^{1-\delta}\geq 100$, $\lambda_{0}\leq 1$ (i.e. $N^{\gamma}\geq\frac{\pi M^{1-\gamma}}{2\tilde{t}C_{\gamma}\gamma}$),
    \item $k\in\mathds{N}$, $\beta\in(0,1]$, $\gamma\in(0,1)$,
    \item $k+\beta>1+2 \delta+\gamma$,
    \item $L<\frac{M}{2}$,
    \item $\frac{\partial^i \frac{v^{pol}_{r,\gamma}(f_{1})}{r}}{\partial r^i}(r=1)=0$ for $i=1,2$,
    \item $\frac{\partial f_{1}}{\partial r}=1$ if $r\in[\frac{3}{4},\frac{5}{4}]$,
    \item $supp(f_{1})\subset \{r:r\in(\frac{1}{2},K_{\gamma})\}$ for some $K_{\gamma}$ depending only on $\gamma$.
\end{itemize}
As before, to avoid extra sub-indexes we consider  $k,\beta,\delta$ and $\gamma$ to be fixed, but all the results will apply as long as they fulfil the restrictions mentioned. The constants appearing in the lemmas might depend on our specific choice but the final results will not.

However it is not immediately obvious whether the conditions we impose over $f_{1,\gamma}$ are too restrictive, so we need the following lemma to assure us that a $f_{1,\gamma}$ with the desired properties exists.

\begin{lemma}\label{valphaarb}
There exists  a $C^{\infty}$  compactly supported function  $g(.):[0,\infty)\rightarrow\mathds{R}$ with support in $ (2,\infty)$ such that $\frac{\partial^{i}\frac{ v_{\alpha,\gamma}(g(.))(r)}{r}}{\partial r^{i}}(r=1)=a_{i}$ with $i=1,2$ and $a_{i}$ arbitrary.
\end{lemma}

We will omit the proof of this lemma since it is completely equivalent to that of lemma 2.5 in \cite{Zoroacordoba}. With this, the existence of the desired $f_{1}$ is easy to prove, since we can just choose some $C^{\infty}$ $\tilde{f}$ with support in $(\frac{1}{2},2)$ with the desired derivative in $r\in[\frac34,\frac54]$ and then add some other $C^{\infty}$ function given by lemma \ref{valphaarb} to cancel out the derivatives of $V_{\alpha,\gamma}$ around $r=1$.

Our next goal will be to prove that this family of pseudo-solutions is a good approximation for our solutions. For this  we define $\bar{v}_{r,\gamma}$ as

$$\bar{v}^{pol}_{r,\gamma}(f_{2}(N^{1-\delta}(r-1),N^{1-\delta}\alpha+c_{1} )\cos(NK\alpha+c_{2}))(r,\alpha)$$
$$:=(NK)^{\gamma} C_{\gamma}f_{2}(N^{1-\delta}(r-1),N^{1-\delta}\alpha+c_{1}) \sin (NK\alpha+c_{2}),$$

$$\bar{v}_{r,\gamma}(f(r))=0.$$

We will only use this definition for ease of notation and we will only apply this operator to our pseudo-solution so we do not have to worry about defining this for a more general function.

With this, the evolution equation for $\bar{w}^{pol}_{\lambda,N,M,J,L,\tilde{t}}$
is

\begin{align*}
    &\frac{\partial \bar{w}^{pol}_{\lambda,N,M,J,L,\tilde{t}}}{\partial t}\\
    &= -v^{pol}_{\alpha,\gamma}(\lambda_{0} f_{1})(r=1)\frac{\partial \bar{w}^{pol}_{\lambda,N,M,J,L,\tilde{t}}}{\partial \alpha}-\lambda_{0}\bar{v}^{pol}_{r,\gamma}(\bar{w}^{pol}_{\lambda,N,M,J,L,\tilde{t}})\\
    &=-v^{pol}_{\alpha,\gamma}(\lambda_{0} f_{1})(r=1)\frac{\partial \bar{w}^{pol}_{\lambda,N,M,J,L,\tilde{t}}}{\partial \alpha}-\frac{\partial \lambda_{0} f_{1}(r)}{\partial r}\bar{v}^{pol}_{r,\gamma}(\bar{w}^{pol}_{\lambda,N,M,J,L,\tilde{t}})\\
\end{align*}

while on the other hand, if $w_{\lambda,N,M,J,L,\tilde{t}}$ is the solution to $\gamma$-SQG with the same initial conditions as $\bar{w}_{\lambda,N,M,J,L,\tilde{t}}$ 
then
\begin{align*}
    &\frac{\partial w^{pol}_{\lambda,N,M,J,L,\tilde{t}}}{\partial t}\\
    &=-\frac{v^{pol}_{\alpha,\gamma}(w^{pol}_{\lambda,N,M,J,L,\tilde{t}})}{r}\frac{\partial w^{pol}_{\lambda,N,M,J,L,\tilde{t}}}{\partial \alpha}-\frac{\partial w^{pol}_{\lambda,N,M,J,L,\tilde{t}}}{\partial r}v^{pol}_{r,\gamma}(w^{pol}_{\lambda,N,M,J,L,\tilde{t}})\\
\end{align*}

and we can rewrite the evolution equation of $w^{pol}_{\lambda,N,M,J,L,\tilde{t}}$ in pseudo-solution form as

\begin{align*}
    &\frac{\partial \bar{w}^{pol}_{\lambda,N,M,J,L,\tilde{t}}}{\partial t}\\
    &=-\frac{v^{pol}_{\alpha,\gamma}(\bar{w}^{pol}_{\lambda,N,M,J,L,\tilde{t}})}{r}\frac{\partial \bar{w}^{pol}_{\lambda,N,M,J,L,\tilde{t}}}{\partial \alpha}-\frac{\partial \bar{w}^{pol}_{\lambda,N,M,J,L,\tilde{t}}}{\partial r}v_{r,\gamma}(\bar{w}^{pol}_{\lambda,N,M,J,L,\tilde{t}})-F^{pol}_{\lambda,N,M,J,L,\tilde{t}}\\    
\end{align*}

with

$$F^{pol}_{\lambda,N,M,J,L,\tilde{t}}=F^{pol}_{1}+F^{pol}_{2}+F^{pol}_{3}+F^{pol}_{4},$$

$$F^{pol}_{1}:=\frac{v^{pol}_{\alpha,\gamma}( \lambda_{0} f_{1}-\bar{w}^{pol}_{\lambda,N,M,J,L,\tilde{t}})}{r}\frac{\partial \bar{w}^{pol}_{\lambda,N,M,J,L,\tilde{t}}}{\partial \alpha}$$
$$F^{pol}_{2}:=(v^{pol}_{\alpha,\gamma}( \lambda_{0} f_{1})(r=1)-\frac{v^{pol}_{\alpha,\gamma}( \lambda_{0} f_{1})}{r})\frac{\partial \bar{w}^{pol}_{\lambda,N,M,J,L,\tilde{t}}}{\partial \alpha}$$
$$F^{pol}_{3}:=\frac{\partial (\lambda_{0} f_{1}(r)-\bar{w}^{pol}_{\lambda,N,M,J,L,\tilde{t}})}{\partial r}v_{r,\gamma}(\bar{w}^{pol}_{\lambda,N,M,J,L,\tilde{t}})$$
$$F^{pol}_{4}:=\frac{\partial \lambda_{0} f_{1}(r)}{\partial r}\big(\bar{v}_{r,\gamma}(\bar{w}^{pol}_{\lambda,N,M,J,L,\tilde{t}})-v_{r,\gamma}(\bar{w}^{pol}_{\lambda,N,M,J,L,\tilde{t}})\big).$$

The next step in our proof will be to show that $F_{\lambda,N,M,J,L,\tilde{t}}$ can be made as small as we need by choosing appropriately the parameters, namely we will show that it becomes small as we make $N$ big.

Before we get to prove that, there are some basic properties of $\bar{w}_{\lambda,N,M,J,L,\tilde{t}}$ that we will need later on

\begin{itemize}
    \item $$||\bar{w}^{pol}_{\lambda,N,M,J,L,\tilde{t}}(r,\alpha,t)||_{C^{m,\beta'}}\leq C_{1}+C_{2} \lambda(NM)^{m+\beta'-k-\beta}$$ 
    $$||\bar{w}^{pol}_{\lambda,N,M,J,L,\tilde{t}}(r,\alpha,t)-\lambda_{0} f_{1}(r)||_{C^{m,\beta'}}\leq C_{2} \lambda(NM)^{m+\beta'-k-\beta}$$
    for any $m\in\mathds{N}$, $\beta'\in[0,1]$, $t\in\mathds{R}$, with $C_{1}$ and $C_{2}$ depending on $m$ and $\beta'$.
    \item $$||\bar{w}^{pol}_{\lambda,N,M,J,L,\tilde{t}}(r,\alpha,t)||_{H^{m}}\leq C_{1}+C_{2} \lambda N^{-1+\delta}(NM)^{m-k-\beta}$$
    $$||\bar{w}^{pol}_{\lambda,N,M,J,L,\tilde{t}}(r,\alpha,t)-\lambda_{0} f_{1}(r)||_{H^{m}}\leq C_{2} \lambda N^{-1+\delta}(NM)^{m-k-\beta}$$
    for any $m\in\mathds{N}$, $\beta'\in[0,1]$, $t\in\mathds{R}$, with $C_{1}$ and $C_{2}$ depending on $m$.
    \item $$||\bar{w}_{\lambda,N,M,J,L,\tilde{t}}(x_{1},x_{2},t)||_{C^{m,\beta'}}\leq C_{1}+C_{2} \lambda(NM)^{m+\beta'-k-\beta}$$
    $$||\bar{w}_{\lambda,N,M,J,L,\tilde{t}}(x_{1},x_{2},t)-\lambda_{0} f_{1}(\sqrt{x_{1}^2+x_{2}^{2}})||_{C^{m,\beta'}}\leq C_{2} \lambda(NM)^{m+\beta'-k-\beta}$$ 
    for any $m\in\mathds{N}$, $\beta'\in[0,1]$, $t\in\mathds{R}$, with $C_{1}$ and $C_{2}$ depending on $m$ and $\beta'$.
    \item $$||\bar{w}_{\lambda,N,M,J,L,\tilde{t}}(x_{1},x_{2},t)||_{H^{m}}\leq C_{1}+C_{2} \lambda N^{-1+\delta}(NM)^{m-k-\beta}$$
    $$||\bar{w}_{\lambda,N,M,J,L,\tilde{t}}(x_{1},x_{2},t)-\lambda_{0} f_{1}(\sqrt{x_{1}^2+x_{2}^{2})}||_{H^{m}}\leq C_{2} \lambda N^{-1+\delta}(NM)^{m-k-\beta}$$
    for any $m\in\mathds{N}$, $\beta'\in[0,1]$, $t\in\mathds{R}$, with $C_{1}$ and $C_{2}$ depending on $m$.
     \item  By using the interpolation inequality for sobolev spaces we also have
     $$||\bar{w}_{\lambda,N,M,J,L,\tilde{t}}(x_{1},x_{2},t)-\lambda_{0} f_{1}(x_{1}^2+x_{2}^{2})||_{H^{m}}\leq C_{1} \lambda N^{-1+\delta}(NM)^{m-k-\beta}$$
    for any $m>0$, $t\in\mathds{R}$, with  $C_{1}$ depending on $m$.
\end{itemize}

The bounds in polar coordinates are obtained by direct calculation and then we obtain from those the ones in cartesian coordinates using that the functions are compactly supported and with support far from the origin. Now, for our pseudo-solutions to be a useful approximation of the solution to $\gamma$-SQG, we need the source term to be small. For that we have the following lemmas.
\begin{lemma}
For any fixed $T$, if $0\leq t\leq T$  we have that

$$||F_{\lambda,N,M,J,L,\tilde{t}}||_{L^{2}}\leq  (1+\frac{1}{\tilde{t}})\frac{C}{N^{k+\beta+1}}$$
with $C$ depending on $T,\lambda,M,J$ and $L$.

Furthermore, for $m\in\mathds{N}$, we have that

$$||F_{\lambda,N,M,J,L,\tilde{t}}||_{H^{m}}\leq C(1+\frac{1}{\tilde{t}})\frac{N^{m}}{N^{k+\beta+1}}$$
with $C_{m}$ depending on $T,\lambda,M,J,L$ and $m$. In fact, by interpolation, the inequality also holds for any $m>0$.

\end{lemma}

\begin{proof}
We start by obtaining bounds for $||F_{1}||_{L^{2}}$. We have that
\begin{align*}
    &||F_{1}||_{L^{2}}\leq ||v_{\alpha,\gamma}( \lambda_{0} f_{1}-\bar{w}_{\lambda,N,M,J,L,\tilde{t}})1_{|x|\geq \frac12}||_{L^{2}}||\frac{1}{r}\frac{\partial \bar{w}_{\lambda,N,M,J,L,\tilde{t}}}{\partial \alpha}||_{L^{\infty}}\\
    &\leq C ||\lambda_{0} f_{1}-\bar{w}_{\lambda,N,M,J,L,\tilde{t}}||_{H^{\gamma}}||\frac{\partial \bar{w}_{\lambda,N,M,J,L,\tilde{t}}}{\partial \alpha}||_{L^{\infty}}\\
    &\leq C \frac{N^{\gamma}}{N^{k+\beta+1-\delta}}\frac{1}{N^{k+\beta-1}}\leq C \frac{1}{N^{k+\beta +1}}.\\
\end{align*}

For $F_{2}$, using that the first two derivatives with respect to $r$ of $\frac{v_{\alpha,\gamma}(\lambda_{0}f_{1})}{r}$ vanish at $r=1$ plus the fact that it is a radial function, we have that if $x\in supp(\frac{\partial \bar{w}_{\lambda,N,M,J,L,\tilde{t}}}{\partial \alpha})$ then 
$$|v^{pol}_{\alpha,\gamma}( \lambda_{0} f_{1})(r=1)-\frac{v^{pol}_{\alpha,\gamma}( \lambda_{0} f_{1})}{r}|\leq C N^{-3+3\delta}$$ 
so

\begin{align*}
    &||F^{pol}_{2}||_{L^{2}}\leq C\lambda_{0}N^{-3+3\delta}||\frac{\partial \bar{w}^{pol}_{\lambda,N,M,J,L,\tilde{t}}}{\partial \alpha}||_{L^{2}}\leq C\frac{N^{-3+4\delta}}{N^{k+\beta}}\leq \frac{C}{N^{k+\beta+1}}\\
\end{align*}

Similarly, for $F_{3}$ we have

\begin{align*}
    &||F_{3}||_{L^{2}}\leq ||\frac{\partial (\bar{w}^{pol}_{\lambda,N,M,J,L,\tilde{t}}-\lambda_{0} f_{1}(r))}{\partial r}||_{L^{\infty}}||v_{r,\gamma}(\bar{w}^{pol}_{\lambda,N,M,J,L,\tilde{t}})1_{|x|\geq \frac12}||_{L^{2}}\\
    &\leq ||\frac{\partial (\bar{w}^{pol}_{\lambda,N,M,J,L,\tilde{t}}-\lambda_{0} f_{1}(r))}{\partial r}||_{L^{\infty}}||v_{r,\gamma}(\bar{w}^{pol}_{\lambda,N,M,J,L,\tilde{t}}-\lambda_{0} f_{1}(r))||_{H^{\gamma}}\\
    &\leq C \frac{N^{1-\delta}}{N^{k+\beta}} \frac{N^{\gamma}}{N^{k+\beta+1-\delta}}\leq C \frac{1}{N^{k+\beta+1 }}.\\
\end{align*}

Finally, for $F_{4}$, we go back to cartesian coordinates and divide the integral in two different parts, $A_{1}:=B_{2N^{-1+\delta}}(\cos(t\lambda_{0} v_{\alpha,\gamma}(f_{1})(r=1)),\sin(t\lambda_{0} v_{\alpha,\gamma}(f_{1})(r=1)))$ $A_{2}:=supp(\bar{w}_{\lambda,N,M,J,L,\tilde{t}})\setminus A_{1}$ we have

\begin{align*}
    &||F_{4}||_{L^{2}}\\
    &\leq ||\frac{\partial \lambda_{0} f_{1}}{\partial r}\big(\bar{v}_{r,\gamma}(\bar{w}_{\lambda,N,M,J,L,\tilde{t}})-v_{r,\gamma}(\bar{w}_{\lambda,N,M,J,L,\tilde{t}})\big)1_{A_{1}}||_{L^{2}}\\
    &+ ||\frac{\partial \lambda_{0} f_{1}}{\partial r}\big(\bar{v}_{r,\gamma}(\bar{w}_{\lambda,N,M,J,L,\tilde{t}})-v_{r,\gamma}(\bar{w}_{\lambda,N,M,J,L,\tilde{t}})\big)1_{A_{2}}||_{L^{2}}.\\
\end{align*}

For the bound on $A_{1}$, using lemma \ref{errorvsmall} and $|\frac{\partial  f_{1}}{\partial r}|\leq C$ we get

\begin{align*}
    &||\frac{\partial \lambda_{0} f_{1}(r)}{\partial r}\big(\bar{v}_{r,\gamma}(\bar{w}_{\lambda,N,M,J,L,\tilde{t}})-v_{r,\gamma}(\bar{w}_{\lambda,N,M,J,L,\tilde{t}})\big)1_{A_{1}}||_{L^{2}}\\
    &\leq || \lambda_{0} f_{1}(r)||_{C^1}||\bar{v}_{r,\gamma}(\bar{w}_{\lambda,N,M,J,L,\tilde{t}})-v_{r,\gamma}(\bar{w}_{\lambda,N,M,J,L,\tilde{t}})\big)1_{A_{1}}||_{L^{\infty}}|A_{1}|^{\frac12}\\
    &\leq C\lambda_{0} \frac{N^{\gamma-\delta}}{N^{k+\beta+1-\delta}}= C\frac{1}{\tilde{t}N^{k+\beta+1}}\\
\end{align*}
where we used that $\lambda_{0}=\frac{CN^{-\gamma}}{\tilde{t}}$ (the constant $C$ depending on $M$).

For the integral in $A_{2}$ using lemma \ref{decaimiento} and the bounds on $f_{1}$ we have

$$\Big(\int_{A_{2}}(\frac{\partial \lambda_{0} f_{1}}{\partial r}\big(\bar{v}_{r,\gamma}(\bar{w}_{\lambda,N,M,J,L,\tilde{t}})-v_{r,\gamma}(\bar{w}_{\lambda,N,M,J,L,\tilde{t}})\big))^{2}dx_{1}dx_{2}\Big)^{\frac12}$$
$$\leq \tilde{t}^{-1}C N^{-\gamma} \Big(\int_{A_{2}}(v_{r,\gamma}(\bar{w}_{\lambda,N,M,J,L,\tilde{t}}))^{2}dx_{1}dx_{2}\Big)^{\frac12}$$
$$\leq \tilde{t}^{-1}\frac{C}{N^{k+\beta+\gamma}} \Big(\int_{2N^{-1+\delta}}^{\infty}\bigg(N^{-2+\delta}\frac{C}{h^{2+\gamma}}\bigg)^{2}hdh\Big)^{\frac12}$$
$$\leq \frac{C}{\tilde{t}N^{k+\beta+1}} $$

For the proof for the bound in $H^{m}$, we use the that, since $supp(w^{pol}_{\lambda,N,M,J,L,\tilde{t}})\subset\{(r,\alpha): r\in [\frac12,K]\}$ for some $K$, then

$$||w_{\lambda,N,M,J,L,\tilde{t}}||_{H^{m}}\leq ||w^{pol}_{\lambda,N,M,J,L,\tilde{t}}||_{H^{m}}$$
and therefore we just need to find bound for

$$\sum_{k=0}^{m}\sum_{j=0}^{k}||\frac{\partial^{k}F_{i}}{\partial^{j}r\partial^{k-j}\alpha}||_{L^{2}}$$
with $i=1,2,3,4$.

For the bounds in $H^{m}$ we will use that, given two functions $f,g$ and $m\in\mathds{Z}$ we have

$$||fg||_{H^{m}}\leq C \sum_{i=0}^{m}||f||_{C^{i}}||g||_{H^{m-i}}$$
with $C$ depending on $m$. Combining this with (\ref{dercomalpha}) we have

\begin{align*}
    &||F_{1}||_{H^{m}}\leq C\sum_{i=0}^{m}||v_{\alpha,\gamma}( \lambda_{0} f_{1}-\bar{w}^{pol}_{\lambda,N,M,J,L,\tilde{t}})1_{|x|\geq \frac12}||_{H^{i}}||\frac{1}{r}\frac{\partial \bar{w}^{pol}_{\lambda,N,M,J,L,\tilde{t}}}{\partial \alpha}||_{C^{m-i}}\\
    &\leq C \sum_{i=0}^{m}\frac{N^{i}N^{-1+\delta+\gamma}}{N^{k+\beta}}\frac{N^{m-i+1}}{N^{k+\beta}}\leq C\frac{N^{m}}{N^{k+\beta+1}}.\\
\end{align*}

For $F_{2}$, using that, for $r\in B^{pol}:=supp(\frac{\partial \bar{w}^{pol}_{\lambda,N,M,J,L,\tilde{t}}}{\partial \alpha})$ we have that 
$$\frac{\partial^{i}(v^{pol}_{\alpha,\gamma}( \lambda_{0} f_{1})(r=1)-\frac{v^{pol}_{\alpha,\gamma}( \lambda_{0} f_{1})}{r})}{\partial r^{i}}\leq C N^{(3-i)(-1+\delta)} $$
for $i=0,1,2$, and since $v^{pol}_{\alpha,\gamma}( \lambda_{0} f_{1})(r=1)-\frac{v^{pol}_{\alpha,\gamma}( \lambda_{0} f_{1})}{r}$ only depends on $r$, then, for $i=0,1,2$
$$||v_{\alpha,\gamma}( \lambda_{0} f_{1})(r=1)-\frac{v^{pol}_{\alpha,\gamma}( \lambda_{0} f_{1})}{r}1_{x\in B}||_{C^{i}}\leq CN^{(3-i)(-1+\delta)}$$
and for higher derivatives we just use
$$||v_{\alpha,\gamma}( \lambda_{0} f_{1})(r=1)-\frac{v^{pol}_{\alpha,\gamma}( \lambda_{0} f_{1})}{r}1_{x\in B}||_{C^{i}}\leq C,$$
where the constant depends on $i$. With this we get

\begin{align*}  
&||F_{2}||_{H^{m}}\leq C \sum_{i=0}^{m}||v_{\alpha,\gamma}( \lambda_{0} f_{1})(r=1)-\frac{v_{\alpha,\gamma}( \lambda_{0} f_{1})}{r} 1_{x\in B}||_{C^{i}}||\frac{\partial \bar{w}_{\lambda,N,M,J,L,\tilde{t}}}{\partial \alpha}||_{H^{m-i}}\\
&\leq C \frac{N^{-3+4\delta+m}}{\tilde{t}N^{k+\beta}}\leq \frac{C N^{m}}{\tilde{t}N^{k+\beta+1}}.\\
\end{align*}

For $F_{3}$ we have

\begin{align*}
    &||F_ {3}||_{H^{m}}\leq C\sum_{i=0}^{m}||\frac{\partial (\bar{w}_{\lambda,N,M,J,L,\tilde{t}}-\lambda_{0} f_{1}(r))}{\partial r}||_{C^{i}}||v_{r,\gamma}(\bar{w}_{\lambda,N,M,J,L,\tilde{t}})||_{H^{m-i}}\\
    &\leq C\sum_{i=0}^{m}\frac{N^{i+1}}{N^{k+\beta}}\frac{N^{m-i-1+\delta+\gamma}}{N^{k+\beta}}\leq C \frac{N^{m}}{N^{k+\beta+1}}.\\
\end{align*}

As for $F_{4}$, the contribution obtained when integrating in $A_{2}$ is obtained again applying lemma \ref{decaimiento}
\begin{align*}
    &||F_{4}1_{A_{2}}||_{H^{m}}\\
    &\leq C\sum_{i} || \lambda_{0} f_{1}(r)||_{C^{i+1}}||\bar{v}_{r,\gamma}(\bar{w}_{\lambda,N,M,J,L,\tilde{t}})-v_{r,\gamma}(\bar{w}_{\lambda,N,M,J,L,\tilde{t}})\big)1_{A_{1}}||_{C^{m-i}}|A_{1}|^{\frac12}\\
    &\leq C\lambda_{0} \frac{N^{m+\gamma-\delta}}{N^{k+\beta+1-\delta}}= C\frac{N^{m}}{\tilde{t}N^{k+\beta+1}}.\\
\end{align*}

 For the contribution when we integrate $F_{4}$ over $A_{1}$ using (\ref{dercomr}) we have

\begin{align*}
    &||F_{4} 1_{x\in A_{1}}||_{H^{m}}\leq C\lambda_{0} ||(\bar{v}_{r,\gamma}(\bar{w}_{\lambda,N,M,J,L,\tilde{t}})-v_{r,\gamma}(\bar{w}_{\lambda,N,M,J,L,\tilde{t}}))1_{x\in A_{1}}||_{H^{m}}\\
    &\leq C\lambda_{0} \sum_{q=0}^{m}\sum_{j=0}^{q}||\Big(\frac{\partial^{q} \bar{v}_{r,\gamma}(\bar{w}_{\lambda,N,M,J,L,\tilde{t}})}{\partial x_{1}^{j}\partial x_{2}^{q-j}}-v_{r,\gamma}(\frac{\partial^{q} \bar{w}_{\lambda,N,M,J,L,\tilde{t}}}{\partial x_{1}^{j}\partial x_{2}^{q-j}})\Big) 1_{x\in A_{1}}||_{L^{2}}\\
    &+C\lambda_{0}||v_{1,\gamma}(\bar{w}_{\lambda,N,M,J,L,\tilde{t}})1_{x\in A_{1}}||_{H^{m-1}}+C\lambda_{0}||v_{2,\gamma}(\bar{w}_{\lambda,N,M,J,L,\tilde{t}})1_{x\in A_{1}}||_{H^{m-1}}\\
    &\leq C\lambda_{0} \sum_{q=0}^{m}\sum_{j=0}^{q}||\Big(\frac{\partial^{q} \bar{v}_{r,\gamma}(\bar{w}_{\lambda,N,M,J,L,\tilde{t}})}{\partial x_{1}^{j}\partial x_{2}^{q-j}}-v_{r,\gamma}(\frac{\partial^{q} \bar{w}_{\lambda,N,M,J,L,\tilde{t}}}{\partial x_{1}^{j}\partial x_{2}^{q-j}})\Big) 1_{x\in A_{1}}||_{L^{2}}\\
    &+C\frac{N^{-1+\delta}}{\tilde{t}N^{k+\beta}}N^{m-1}\\
\end{align*}

But then since

$$\frac{\partial^{q}f(r,\alpha)}{\partial x_{1}^{j}\partial x_{2}^{q-j}}=\sum_{p=0}^{q}\sum_{l=0}^{p}g_{q,j,p,l}(r,\alpha)\frac{\partial^{p}f(r,\alpha)}{\partial r^{l}\partial \alpha^{p-l}}$$
with $g_{m,j,q,l}$ in $C^{\infty}$ and bounded if $r\geq \frac12$, we have that

\begin{align*}
    &||\Big(\frac{\partial^{q} \bar{v}_{r,\gamma}(\bar{w}_{\lambda,N,M,J,L,\tilde{t}})}{\partial x_{1}^{j}\partial x_{2}^{q-j}}-v_{r,\gamma}(\frac{\partial^{q} \bar{w}_{\lambda,N,M,J,L,\tilde{t}}}{\partial x_{1}^{j}\partial x_{2}^{q-j}})\Big) 1_{x\in A_{1}}||_{L^{2}}\\
    \leq& \sum_{p=0}^{q}\sum_{l=0}^{p}||\Big(g_{q,j,p,l}(r,\alpha)\frac{\partial^{p}\bar{v}^{pol}_{r,\gamma}(\bar{w}_{\lambda,N,M,J,L,\tilde{t}})}{\partial r^{l}\partial \alpha^{p-l}}-v^{pol}_{r,\gamma}\big(g_{q,j,p,l}(r,\alpha)\frac{\partial^{p} \bar{w}_{\lambda,N,M,J,L,\tilde{t}}}{\partial r^{l}\partial \alpha^{p-l}}\big)\Big)1_{(r,\alpha)\in A^{pol}_{1}}||_{L^{2}}.\\
\end{align*}

But applying lemma \ref{errorvsmall} to each of the terms we obtain after differentiating, we get

\begin{align*}
    \leq& \sum_{p=0}^{q}\sum_{l=0}^{p}||\Big(g_{q,j,p,l}(r,\alpha)\frac{\partial^{p}\bar{v}^{pol}_{r,\gamma}(\bar{w}_{\lambda,N,M,J,L,\tilde{t}})}{\partial r^{l}\partial \alpha^{p-l}}-v^{pol}_{r,\gamma}\big(g_{q,j,p,l}(r,\alpha)\frac{\partial^{p} \bar{w}_{\lambda,N,M,J,L,\tilde{t}}}{\partial r^{l}\partial \alpha^{p-l}}\big)\Big)1_{(r,\alpha)\in A^{pol}_{1}}||_{L^{2}}\\
    &\leq  C\frac{N^{q}N^{\gamma-\delta}N^{-1+\delta}}{N^{k+\beta}}\\
\end{align*}
so

\begin{align*}
    &||F_{4} 1_{x\in A_{1}}||_{H^{m}}\leq C\frac{N^{-1+\delta}}{\tilde{t}N^{k+\beta}}N^{m-1}+C\lambda_{0} \sum_{q=0}^{m}\sum_{j=0}^{q}\frac{N^{q}N^{\gamma-\delta}N^{-1+\delta}}{N^{k+\beta}}\leq \frac{C N^{m}}{\tilde{t} N^{k+\beta+1}}\\
\end{align*}
and we are done.
\end{proof}

Since we are interested in showing (arbitrarily) fast norm growth for $\gamma$-SQG, our solution should start with a very small norm that gets very big after a short period of time. Lemma \ref{decaimiento} already gives us tools to show that the initial norm is small, and the next lemma will gives us a lower bound for the $C^{k,\beta}$ norm of our pseudo-solutions at time $\tilde{t}$.

\begin{lemma}\label{crecimientockbeta}
There exists a set $A$ (depending on $\lambda,N,M,J$ and $L$) such that, if $x\in A$ then there exists unitary $u$ depending on $x$  and a constant $C$ with

$$|\frac{\partial^{k}(\bar{w}_{\lambda,N,M,J,L,\tilde{t}}(x,\tilde{t})-\lambda_{0}f_{1})}{\partial u^{k}}|\geq  \lambda( \frac{1 }{2(MN)^{\beta}}-\frac{CL^{2}  }{(NM)^{\beta}M}-C(NM)^{-(\delta+\beta)}-C(NM)^{-\beta}N^{-1+\delta})  $$

and a set $B$ (depending on $\lambda,N,M,J$ and $L$) such that if $x\in B$ then for all unitary $v$ we have that
$$|\frac{\partial^{k}(\bar{w}_{\lambda,N,M,J,L,\tilde{t}}(x,\tilde{t})-\lambda_{0}f_{1})}{\partial u^{k}}|\leq  \lambda( \frac{1 }{4(MN)^{\beta}}+ C(NM)^{-(\delta+\beta)}+\frac{CL^{2}  }{(NM)^{\beta}M})$$
furthermore, there is a set $S_{M,N,\delta}$ with $|S_{M,N,\delta}|\geq C_{1}MN^{2\delta}$,

$$A=\cup_{s\in S_{M,N,\delta}} A_{s}$$

$$B=\cup_{s\in S_{M,N,\delta}}B_{s}$$
$d(x,y)\leq \frac{4\pi}{NM}$ if $x\in A_{s}, y\in B_{s}$,  and $|A_{s}|,|B_{s}|\geq \frac{C_{2}}{(NM)^2}$, with $C_{1}$ and $C_{2}$ constants.

Note that, in particular 

$$||\bar{w}_{\lambda,N,M,J,L,\tilde{t}}(x,\tilde{t})-\lambda_{0}f_{1}||_{C^{k,\beta}}\geq\lambda(   \frac{1}{4(4\pi)^{\beta}} -\frac{CL^{2}  }{M}-C(NM)^{-\delta}-CN^{-1+\delta}) $$

\end{lemma}

\begin{proof}

We start by finding the set $A$ as well as the unitary vector $v$ that gives us a big $k-th$ derivative.  

For this, we first want to obtain accurate estimates for $\frac{\partial^{k}( \bar{w}^{pol}_{\lambda,N,M,J,L,\delta,t}(r,\alpha,t)-\lambda_{0}f_{1}) }{\partial \alpha^{k} }$

\begin{align*}
    &\frac{\partial^{k} \bar{w}^{pol}_{\lambda,N,M,J,L,\tilde{t}}(r,\alpha,t) }{\partial \alpha^{k} }=\frac{\partial^{k} (\bar{w}^{pol}_{\lambda,N,M,J,L,\tilde{t}}(r,\alpha,t)-\lambda_{0}f_{1}) }{\partial \alpha^{k} }\\
    &=\sum_{i=0}^{k}{k\choose i}\sum_{j=1}^{J}\sum_{l=0}^{L-1}\bigg(\frac{1}{JL(NM_{j})^{k+\beta}}\frac{\partial^{i}\lambda_{j}f_{2}(N^{1-\delta}(r-1),N^{1-\delta}(\alpha-t\lambda_{0} v_{\alpha,\gamma}(f_{1})(r=1)))}{\partial \alpha^{i}}\\ \nonumber
    & \frac{\partial^{k-i}\cos(N(M_{j}+l)(\alpha-\alpha_{j}^{1}-t\lambda_{0} v_{\alpha,\gamma}(f_{1})(r=1))+\alpha_{j}^{2}+\frac{k\pi}{2}+t\lambda_{0} C_{\gamma}N^{\gamma}(M_{j}+l)^{\gamma}}{\partial \alpha^{k-i}}\big),\\
\end{align*}

and so

\begin{align*}
    &|\frac{\partial^{k} \bar{w}^{pol}_{\lambda,N,M,J,L,\tilde{t}}(r,\alpha,t) }{\partial \alpha^{k} }\\
    &-\sum_{j=1}^{J}\sum_{l=0}^{L-1}\bigg(\frac{1}{JL(NM_{j})^{k+\beta}}\lambda_{j}f_{2}(N^{1-\delta}(r-1),N^{1-\delta}(\alpha-t\lambda_{0} v_{\alpha,\gamma}(f_{1})(r=1)))\\ \nonumber
    & \frac{\partial^{k}\cos(N(M_{j}+l)(\alpha-\alpha_{j}^{1}-t\lambda_{0} v_{\alpha,\gamma}(f_{1})(r=1))+\alpha_{j}^{2}+\frac{k\pi}{2}+t\lambda_{0} C_{\gamma}N^{\gamma}(M_{j}+l)^{\gamma}}{\partial \alpha^{k})}\bigg)|\\
    &\leq C \lambda (NM)^{-(\delta+\beta)}.\\
\end{align*}

Furthermore

\begin{align*}
    &\frac{\partial^{k}\cos(N(M_{j}+l)(\alpha-\alpha_{j}^{1}-t\lambda_{0} v_{\alpha,\gamma}(f_{1})(r=1))+\alpha_{j}^{2}+\frac{k\pi}{2}+t\lambda_{0} C_{\gamma}N^{\gamma}(M_{j}+l)^{\gamma})}{\partial \alpha^{k}}\\
    &=(N(M_{j}+l))^{k} \cos(N(M_{j}+l)(\alpha-\alpha_{j}^{1}-t\lambda_{0} v_{\alpha,\gamma}(f_{1})(r=1))+\alpha_{j}^{2}+t\lambda_{0} C_{\gamma}N^{\gamma}(M_{j}+l)^{\gamma})\\ 
    &=(N(M_{j}+l))^{k}\cos(N(M_{j}+l)(\alpha-\alpha_{j}^{1}(t))+\alpha_{j}^{2}(t)+\alpha^{3}_{j,l}(t))\\
\end{align*}

with 
$$\alpha_{j}^{1}(t):=\alpha_{j}^{1}-t\lambda_{0}C_{\gamma}\gamma(NM_{j})^{\gamma-1}+t\lambda_{0} v_{\alpha,\gamma}(f_{1})(r=1)$$
$$\alpha_{j}^{2}(t):=\alpha_{j}^{2}+(1-\gamma)t\lambda_{0}C_{\gamma}(NM_{j})^{\gamma}$$

$$\alpha_{j,l}^{3}(t)=t\lambda_{0}C_{\gamma}((N(M_{j}+l))^{\gamma}-(NM_{j})^{\gamma}-\gamma lN^{\gamma}M_{j}^{\gamma-1}),$$

and we have

\begin{align*}
    &|\cos(N(M_{j}+l)(\alpha-\alpha_{j}^{1}(t))+\alpha_{j}^{2}(t)+\alpha^{3}_{j,l}(t))-\cos(N(M_{j}+l)(\alpha-\alpha_{j}^{1}(t))+\alpha_{j}^{2}(t))|\\
    &\leq C|\alpha^{3}_{j,l}(t)|\leq Ct\gamma(1-\gamma)\lambda_{0}C_{\gamma}(NM_{j})^{\gamma}\frac{L^2}{(M_{j})^2}=\frac{CtJ L^{2}  }{\tilde{t}M_{j}}, \\
\end{align*}

so

\begin{align*}
    &|\frac{\partial^{k} \bar{w}^{pol}_{\lambda,N,M,J,L,\tilde{t}}(r,\alpha,t) }{\partial \alpha^{k} }\\
    &-\sum_{j=1}^{J}\sum_{l=0}^{L-1}\bigg(\frac{1}{JL(NM)^{\beta}}\lambda f_{2}(N^{1-\delta}(r-1),N^{1-\delta}(\alpha-t\lambda_{0} v_{\alpha,\gamma}(f_{1})(r=1)))\\
    & \cos(N(M_{j}+l)(\alpha-\alpha_{j}^{1}(t))+\alpha_{j}^{2}(t))\bigg)|\leq C\lambda (NM)^{-(\delta+\beta)}+\lambda \frac{Ct L^{2}  }{\tilde{t}(NM)^{\beta}M}.\\
\end{align*}

But  we have that $\alpha_{j}^{1}(\tilde{t})=0$, $\alpha_{j}^{2}(\tilde{t})=0$, so that if $\alpha=i\frac{2\pi}{NM}$, $i\in\mathds{Z}$, then

\begin{align*}
    &|\frac{\partial^{k} \bar{w}^{pol}_{\lambda,N,M,J,L,\tilde{t}}(r,\alpha,\tilde{t}) }{\partial \alpha^{k} }-\frac{1}{(NM)^{\beta}}\lambda f_{2}(N^{1-\delta}(r-1),N^{1-\delta}(\alpha-\lambda_{0}\tilde{t} v_{\alpha,\gamma}(f_{1})(r=1)))|\\
    & \leq C\lambda(NM)^{-(\delta+\beta)}+\lambda \frac{CL^{2}  }{(NM)^{\beta}M},\\
\end{align*}
and in fact, if $\alpha\in[i\frac{2\pi }{N}-\frac{\pi}{16NM},i\frac{2\pi }{N}+\frac{\pi}{16NM}]$  with $i\in\mathds{Z}$ then

\begin{align*}
    &\frac{\partial^{k} \bar{w}^{pol}_{\lambda,N,M,J,L,\tilde{t}}(r,\alpha,\tilde{t}) }{\partial \alpha^{k} }\\
    & \geq \frac{1}{2(NM)^{\beta}}\lambda f_{2}(N^{1-\delta}(r-1),N^{1-\delta}(\alpha-\lambda_{0}\tilde{t} v_{\alpha,\gamma}(f_{1})(r=1)))-C\lambda(NM)^{-(\delta+\beta)}-\lambda\frac{C L^{2}  }{(NM)^{\beta}M}.\\
\end{align*}

But since  $f(N^{1-\delta}(r-1),N^{1-\delta}(\alpha-t\lambda_{0} v_{\alpha,\gamma}(f_{1})(r=1)))=1$ if $(r,\alpha)\in[1-\frac{N^{-1+\delta}}{4},1+\frac{N^{-1+\delta}}{4}]\times[t\lambda_{0} v_{\alpha,\gamma}(f_{1})(r=1)-\frac{N^{-1+\delta}}{4},t\lambda_{0} v_{\alpha,\gamma}(f_{1})(r=1)+\frac{N^{-1+\delta}}{4}]$ then defining

\begin{align*}
    &A^{pol}=\cup_{j=-\lfloor \frac{N^{\delta}M}{4}\rfloor}^{j=\lfloor \frac{N^{\delta}M}{4}\rfloor-1}\cup_{i=\lfloor\frac{-N^{\delta}}{64}+\frac{Nt\lambda_{0} v_{\alpha,\gamma}(f_{1})(r=1)}{2\pi}\rfloor}^{i=\lfloor\frac{N^{\delta}}{64}+\frac{Nt\lambda_{0} v_{\alpha,\gamma}(f_{1})(r=1)}{2\pi}\rfloor} A^{pol}_{i,j}\\
\end{align*}
with 
$$A_{i,j}:= (1+\frac{j}{NM},1+\frac{j+1}{NM}]\times[i\frac{2\pi }{N}-\frac{\pi}{16NM},i\frac{2\pi }{N}+\frac{\pi}{16NM}]$$
we have that, for $(r,\alpha)\in A^{pol}$,

\begin{align*}
    &\frac{\partial^{k} \bar{w}^{pol}_{\lambda,N,M,J,L,\tilde{t}}(r,\alpha,\tilde{t}) }{\partial \alpha^{k} }\\
    & \geq \frac{\lambda}{2(NM)^{\beta}} -C\lambda(NM)^{-(\delta+\beta)}-\lambda \frac{CL^{2}  }{(NM)^{\beta}M}.\\
\end{align*}

Furthermore, the sets $A_{i,j}$ fulfil $|A_{i,j}|\geq C (NM)^{-2}$ for some $C>0$.

Therefore, if we prove that there exists a unitary vector $u=(u_{1},u_{2})$ such that, if $x=(r\cos(\alpha),r \sin(\alpha))\in A$

\begin{align*}
   &\frac{\partial^{k}(\bar{w}_{\lambda,N,M,J,L,\tilde{t}}(x,\tilde{t})-\lambda_{0}f_{1})}{\partial u^{k}} \approx \frac{\partial^{k}\bar{w}^{pol}_{\lambda,N,M,J,L,\tilde{t}}(r,\alpha,\tilde{t})}{\partial \alpha^{k}}\\
\end{align*}
in a suitable way, then we are done proving the existence of the desired set $A$. But

\begin{align*}
    &\frac{\partial f(x)}{\partial u}=u_{1}[\cos(\alpha(x))\frac{\partial f^{pol}(r(x),\alpha(x)) }{\partial  r}-\frac{\sin(\alpha(x))}{r}\frac{\partial f^{pol}(r(x),\alpha(x)) }{\partial  \alpha}]\\
    &+u_{2}[\sin(\alpha(x))\frac{\partial f^{pol}(r(x),\alpha(x)) }{\partial  r}+\frac{\cos(\alpha(x))}{r}\frac{\partial f^{pol}(r(x),\alpha(x)) }{\partial  \alpha}]\\
\end{align*}
so that

\begin{align*}
    &\frac{\partial^{k} f(x)}{\partial u^{k}}=\sum_{i_{1}=0}^{k}\sum_{i_{2}=0}^{i_{1}}g_{i_{1},i_{2}}(\alpha,r,u_{1},u_{2})\frac{\partial f^{pol}(r,\alpha)}{\partial r^{i_{2}}\partial^{i_{1}-i_{2}} \alpha}\\
\end{align*}
with $g_{i_{1},i_{2}}$ $C^{\infty}$ and bounded as long as we only consider $r\geq \frac12$.

Applying this formula to $\bar{w}_{\lambda,N,M,J,L,\tilde{t}}$ we get

\begin{align*}
   &|\frac{\partial^{k}(\bar{w}_{\lambda,N,M,J,L,\tilde{t}}(x,\tilde{t})-\lambda_{0}f_{1})}{\partial u^{k}}-g_{k,0}(\alpha,r,u_{1},u_{2})\frac{\partial^{k}\bar{w}^{pol}_{\lambda,N,M,J,L,\tilde{t}}(r,\alpha,\tilde{t})}{\partial \alpha^{k}}|\\
   &\leq  C\lambda (NM)^{-(\delta+\beta)}\\
\end{align*}
and it is easy to prove that $g_{k,0}=\frac{g_{k-1,0}(\cos(\alpha)u_{2}-\sin(\alpha)u_{1})}{r}$, $g_{0,0}=1$
and therefore taking $v=(-\sin(\alpha),\cos(\alpha))$ we get

\begin{align*}
   &|\frac{\partial^{k}(\bar{w}_{\lambda,N,M,J,L,\tilde{t}}(x,\tilde{t})-\lambda_{0}f_{1})}{\partial u^{k}}-\frac{1}{r^{k}}\frac{\partial^{k}\bar{w}^{pol}_{\lambda,N,M,J,L,\tilde{t}}(r,\alpha,\tilde{t})}{\partial \alpha^{k}}|\\
   &\leq C\lambda (NM)^{-(\delta+\beta)}\\
\end{align*}
and using $r\in A^{pol}\Rightarrow r\in[1-\frac{N^{-1+\delta}}{4},1+\frac{N^{-1+\delta}}{4}]$ plus the bounds for $w^{pol}_{\lambda,N,M,J,L,\tilde{t}}$
\begin{align*}
   &|\frac{\partial^{k}(\bar{w}_{\lambda,N,M,J,L,\tilde{t}}(x,\tilde{t})-\lambda_{0}f_{1})}{\partial u^{k}}-\frac{\partial^{k}\bar{w}^{pol}_{\lambda,N,M,J,L,\tilde{t}}(r,\alpha,\tilde{t})}{\partial \alpha^{k}}|\\
   &\leq C\lambda(NM)^{-(\delta+\beta)}+C\lambda (NM)^{-\beta}N^{-1+\delta},\\
\end{align*}
so, for $x\in A$ 

\begin{align*}
   &\frac{\partial^{k}(\bar{w}_{\lambda,N,M,J,L,\tilde{t}}(x,\tilde{t})}{\partial u^{k}}\\
   &\geq \frac{\lambda}{2(NM)^{\beta}} -\lambda \frac{CL^{2}  }{(NM)^{\beta}M}-C\lambda(NM)^{-(\delta+\beta)}-C\lambda(NM)^{-\beta}N^{-1+\delta},\\
\end{align*}
which finishes the proof for the existence of the set $A$. For the set $B$, we remember that for $r\geq \frac12$ we have

\begin{align*}
   &|\frac{\partial^{k}(\bar{w}_{\lambda,N,M,J,L,\tilde{t}}(x,\tilde{t})-\lambda_{0}f_{1})}{\partial u^{k}}-g_{k,0}(\alpha,r,u_{1},u_{2})\frac{\partial^{k}\bar{w}^{pol}_{\lambda,N,M,J,L,\tilde{t}}(r,\alpha,\tilde{t})}{\partial \alpha^{k}}|\\
   &\leq  C\lambda (NM)^{-(\delta+\beta)}\\
\end{align*}
and since $|g_{k,0}|\leq \frac{1}{r^k}$ we only need to find a sets $B_{i,j}$ with the desired size and distance to $A_{i,j}$ such that $|\frac{\partial^{k}\bar{w}^{pol}_{\lambda,N,M,J,L,\tilde{t}}(r,\alpha,\tilde{t})}{\partial \alpha^{k}}|$ is small. But

\begin{align*}
    &|\frac{\partial^{k}\bar{w}^{pol}_{\lambda,N,M,J,L,\tilde{t}}(r,\alpha,\tilde{t})}{\partial \alpha^{k}}|\\
    &\leq |\sum_{j=1}^{J}\sum_{l=0}^{L-1}\big(\frac{1}{JL(NM)^{\beta}}\lambda f_{2}(N^{1-\delta}(r-1),N^{1-\delta}(\alpha-\lambda_{0}\tilde{t} v_{\alpha,\gamma}(f_{1})(r=1)))\\
    & \cos(N(M_{j}+l)\alpha)\big)|+ C\lambda (NM)^{-(\delta+\beta)}+\lambda \frac{CJL^{2}  }{(NM)^{\beta}M}.\\
\end{align*}
and using

\begin{align*}
    &\sum_{l=0}^{L-1} \cos(N(M_{j}+l)\alpha)= \frac{\sin(L\frac{N\alpha}{2})}{\sin(\frac{N\alpha}{2})}\cos(NM_{j}\alpha+\frac{(L-1)}{2}N\alpha),\\
\end{align*}

\begin{align*}
    &\sum_{j=1}^{J}  \frac{\sin(L\frac{N\alpha}{2})}{\sin(\frac{N\alpha}{2})}\cos(NM\frac{j}{J}\alpha+NM\alpha+\frac{(L-1)}{2}N\alpha)\\
    &=\frac{\sin(L\frac{N\alpha}{2})}{\sin(\frac{N\alpha}{2})}\frac{\sin(\frac{NM\alpha}{2})}{\sin(\frac{NM\alpha}{2J})}\cos(NM(1+\frac{1}{J})\alpha+\frac{(L-1)}{2}N\alpha+\frac{(J-1)NM\alpha}{2J}).\\
\end{align*}

If now we define

$$f^{pol}_{L,N,M,J}(r,\alpha)=\frac{\sin(L\frac{N\alpha}{2})}{\sin(\frac{N\alpha}{2})}\frac{\sin(\frac{NM\alpha}{2})}{\sin(\frac{NM\alpha}{2J})}\cos(NM(1+\frac{1}{J})\alpha+\frac{(L-1)}{2}N\alpha+\frac{(J-1)NM\alpha}{2J})$$

then we have that

\begin{itemize}
    \item $f^{pol}_{L,N,M,J}$ is $\frac{2\pi}{N}-$periodic in the $\alpha$ variable.
    \item There exists $|\tilde{\alpha}|\leq \frac{2\pi}{NM}$ such that $f^{pol}_{L,N,M,J}(r,\tilde{\alpha})=0$.
    \item $|\frac{\partial f^{pol}_{L,N,M,J}(r,\alpha)}{\partial \alpha}|\leq \bar{C}LMNJ$,
\end{itemize}
with $\bar{C}$ a constant, which means that if $\alpha\in \cup_{i\in\mathds{Z}} [\tilde{\alpha}+i \frac{2\pi}{N}-\frac{1}{4\bar{C}MN},\tilde{\alpha}+ i \frac{2\pi}{N}+\frac{1}{4\bar{C}MN}]$ then $|f^{pol}_{L,N,M,J}(r,\alpha)|\leq \frac{JL}{4}$. Using this we have that, if  $\alpha\in \cup_{i\in\mathds{Z}} [\tilde{\alpha}+i \frac{2\pi}{N}-\frac{1}{4\bar{C}MN},\tilde{\alpha}+ i \frac{2\pi}{N}+\frac{1}{4\bar{C}MN}]$ then

\begin{align*}
    &|\frac{\partial^{k}\bar{w}^{pol}_{\lambda,N,M,J,L,\tilde{t}}(r,\alpha,\tilde{t})}{\partial \alpha^{k}}|\\
    & \leq \frac{\lambda}{4(MN)^{\beta}}+ \lambda C(NM)^{-(\delta+\beta)}+\lambda \frac{CJL^{2}  }{(NM)^{\beta}M},\\
\end{align*}
so, for any unitary vector $u$
\begin{align*}
    &|\frac{\partial^{k}\bar{w}_{\lambda,N,M,J,L,\tilde{t}}(x,\tilde{t})}{\partial u^{k}}|\\
    & \leq \frac{\lambda}{4(MN)^{\beta}}+ \lambda C(NM)^{-(\delta+\beta)}+\lambda \frac{CJL^{2}  }{(NM)^{\beta}M},\\
\end{align*}
and defining now

$$B_{i,j}:= (1+\frac{j}{NM},1+\frac{j+1}{NM}]\times[\tilde{\alpha}+i\frac{2\pi }{N}-\frac{\pi}{4\bar{C}NM},\tilde{\alpha}+i\frac{2\pi }{N}+\frac{\pi}{4\bar{C}NM}]$$
and it is easy to check that $A_{i,j}$, $B_{i,j}$ have the desired properties.
\end{proof}

The previous lemma shows that our pseudo-solutions do have a big norm at time $\tilde{t}$, and although this will be enough to show ill-posedness, for our non-existence result we will build solutions such that the $C^{k,\beta}$ norm will be infinite for a period of time, and this requires us to obtain specific bounds about how fast our solution can change their $C^{k,\beta}$ norm.

\begin{lemma}\label{derivadackbeta}
We have that

$$\frac{d||\bar{w}_{\lambda,N,M,J,L,\tilde{t}}(x_{1},x_{2},t)-\lambda_{0}f_{1}(\sqrt{x_{1}^{2}+x_{2}^{2}})||_{C^{k,\beta}}}{dt}\leq \frac{C\lambda M}{\tilde{t}}$$
with $C$ a constant.
\end{lemma}

\begin{proof}
First, since rotations do not change the $C^{k,\beta}$ norm, it is enough to study the evolution of the norm of

\begin{align*}
    & \sum_{j=1}^{J}\sum_{l=0}^{L-1}\bigg(\lambda_{j}f_{2}(N^{1-\delta}(r(x)-1),N^{1-\delta}\alpha(x))\\
    & \frac{\cos(N(M_{j}+l)(\alpha(x)-\alpha_{j}^{1})+\alpha_{j}^{2}+\frac{k\pi}{2}+t\lambda_{0} C_{\gamma}N^{\gamma}(M_{j}+l)^{\gamma}}{JL(NM_{j})^{k+\beta}}\bigg)\\ 
\end{align*}
which has a time derivative

\begin{align*}
    & -\lambda_{0} C_{\gamma}N^{\gamma}(M_{j}+l)^{\gamma}\sum_{j=1}^{J}\sum_{l=0}^{L-1}\bigg(\lambda_{j}f_{2}(N^{1-\delta}(r(x)-1),N^{1-\delta}\alpha(x))\\
    & \frac{\sin(N(M_{j}+l)(\alpha(x)-\alpha_{j}^{1})+\alpha_{j}^{2}+\frac{k\pi}{2}+t\lambda_{0} C_{\gamma}N^{\gamma}(M_{j}+l)^{\gamma}}{JL(NM_{j})^{k+\beta}}\bigg)\\ 
\end{align*}
but since this function has support in $r\geq \frac12$, we can use (\ref{equivcmb}) and it is enough to obtain bounds for the $C^{k,\beta}$ norm in polar coordinates. However, using the expression for $\lambda_{0}$ we easily obtain

\begin{align*}
    & ||\lambda_{0} C_{\gamma}N^{\gamma}(M_{j}+l)^{\gamma}\sum_{j=1}^{J}\sum_{l=0}^{L-1}\bigg(\lambda_{j}f_{2}(N^{1-\delta}(r-1),N^{1-\delta}\alpha)\\
    & \frac{\sin(N(M_{j}+l)(\alpha-\alpha_{j}^{1})+\alpha_{j}^{2}+\frac{k\pi}{2}+t\lambda_{0} C_{\gamma}N^{\gamma}(M_{j}+l)^{\gamma}}{JL(NM_{j})^{k+\beta}}\bigg)||_{C^{k,\beta}}\\
    &\leq \frac{C\lambda M}{\tilde{t}}.\\
\end{align*}

\end{proof}

We only need one last technical result before we can go to prove our ill-posedness result. Namely, we need to obtain bounds for the error between our pseudo-solution and the real solution to $\gamma$-SQG with our initial conditions. We will, however, prove a slightly stronger result, where we show that the error remains small even if we compare to a solution to $\gamma$-SQG with a small error in the velocity. This will later on be necessary when we prove the non-existence of solutions in $C^{k,\beta}.$

\begin{lemma}\label{evolucionerror}
Given a pseudo-solution $\bar{w}_{\lambda,N,M,J,L,\tilde{t}}(x_{1},x_{2},t)$ and a function $v_{error}=(v_{1,error},v_{2,error})$ fulfilling

$$||v_{error}||_{C^{m}}\leq \frac{ N^{m}}{N^{k+\beta+2}}$$

for $m=0,1,...,k+2$ and

$$\frac{\partial v_{1,error}}{\partial x_{1}}+\frac{\partial v_{2,error}}{\partial x_{2}}=0$$

we have that, for any fixed $T$,$\lambda,M,J,L$ and $\tilde{t}$, if $N$ is big enough, then the unique $H^{k+\beta+1-\delta}$ solution $\tilde{w}_{\lambda,N,M,J,L,\tilde{t}}(x_{1},x_{2},t)$ to

\begin{equation}\label{gSQGvext}
   \frac{\partial \tilde{w}_{\lambda,N,M,J,L,\tilde{t}} }{\partial t}+(v_{\gamma}(\tilde{w}_{\lambda,N,M,J,L,\tilde{t}})+v_{error})\cdot(\nabla\tilde{w}_{\lambda,N,M,J,L,\tilde{t}})=0 
\end{equation}

$$\tilde{w}_{\lambda,N,M,J,L,\tilde{t}}(x,0)=\bar{w}_{\lambda,N,M,J,L,\tilde{t}}(x,0)$$
 exists for $t\in[0,T]$ and, if we define 
 $$W:=\tilde{w}_{\lambda,N,M,J,L,\tilde{t}}-\bar{w}_{\lambda,N,M,J,L,\tilde{t}}$$
 then
$$||W(x,t)||_{L^{2}}\leq C(1+\frac{1}{\tilde{t}})t N^{-k-\beta-1},$$
$$||W(x,t)||_{H^{k+\beta+1-\delta}}\leq C(1+\frac{1}{\tilde{t}}) tN^{-\delta}.$$
with $C$ depending on $T,\lambda,M,J$ and $L$.

Furthermore, by interpolation, for any $s\in[0,k+\beta+1-\delta]$ we have that

$$||W(x,t)||_{H^{s}}\leq C(1+\frac{1}{\tilde{t}}) tN^{-(k+\beta+1)+s}.$$
\end{lemma}

\begin{proof}
First we note that the evolution equation for $W$ is
\begin{align*}
    &\frac{\partial W}{\partial t}+(v_{\gamma}(\bar{w}_{\lambda,N,M,J,L,\tilde{t}})+v_{\gamma}(W)+v_{error})\cdot \nabla W\\
    &+(v_{\gamma}(W)+v_{error})\cdot\nabla\bar{w}_{\lambda,N,M,J,L,\tilde{t}}-F_{\lambda,N,M,J,L,\tilde{t}}=0.\\
\end{align*}
and (using the properties of $F_{\lambda,N,M,J,L,\tilde{t}}$ for $N$ big) this evolution equation has local existence and uniqueness in $H^{k+\beta+1-\delta}$ under our assumptions for $v_{error}$. Furthermore, it is enough to prove our inequalities under the assumption $||W(x,t)||_{H^{k+\beta+1-\delta}}\leq C N^{-\delta} log(N)$, since then using the continuity in time of $||W||_{H^{k+\beta+1-\delta}}$ and taking $N$ big would give us the result for the desired time interval.

For the $L^{2}$ norm, we can use incompressibility to obtain

$$\frac{\partial||W||^2_{L^{2}} }{\partial t}\leq 2\int  |W\big(v_{\gamma}(W)\nabla\bar{w}_{\lambda,N,M,J,L,\tilde{t}}-F_{\lambda,N,M,J,L,\tilde{t}}+v_{error}\nabla\bar{w}_{\lambda,N,M,J,L,\tilde{t}}\big)|dx$$
$$\leq \int  2|Wv_{\gamma}(W)\nabla\bar{w}_{\lambda,N,M,J,L,\tilde{t}})dx|+\frac{C}{N^{k+\beta+1}}(1+\frac{1}{\tilde{t}})||W||_{L^2}.$$

To bound the integral term with $v_{\gamma}(W)$ we need to use two important properties that will also be key when working with the $H^{k+\beta+1-\delta}$ bounds. First, as in \cite{Chaecordoba}, using that, for an odd operator $A$ (which in our case will be $v_{1,\gamma}$ and $v_{2,\gamma}$) we have
$$\int fA(f)g=-\frac{1}{2}\int f(A(gf)-gA(f))$$
and so

$$|\int Wv_{\gamma}(W)\nabla\bar{w}_{\lambda,N,M,J,L,\tilde{t}})dx|=\frac{1}{2}|\int  W\big(v_{i,\gamma}(W\frac{\partial  \bar{w}_{\lambda,N,M,J,L,\tilde{t}}}{\partial x_{i}})-v_{i,\gamma}(W)\frac{\partial  \bar{w}_{\lambda,N,M,J,L,\tilde{t}}}{\partial x_{i}}\big)dx|$$
and using corollary 1.4 in \cite{Dongli} 

$$|\int Wv_{\gamma}(W)\nabla\bar{w}_{\lambda,N,M,J,L,\tilde{t}})dx|\leq ||W||^{2}_{L^2}||\nabla v_{\gamma}(\bar{w}_{\lambda,N,M,J,L,\tilde{t}})||_{L^{\infty}}\leq C ||W||^{2}_{L^2}$$
where we used that
$$||v_{\gamma}(\bar{w}_{\lambda,N,M,J,L,\tilde{t}})||_{C^{k',\beta'}}\leq C N^{k'+\beta'+\gamma-k-\beta}log(N)$$
which is obtained by applying lemmas 3.6 and 3.7 from \cite{Zoroacordoba}, the definition of $v_{\gamma}$ and the properties of $\bar{w}_{\lambda,N,M,J,L,\tilde{t}}$.
Then, after applying Gronwall we get
$$||W||_{L^2}\leq \frac{Ct}{N^{k+\beta+1}}(1+\frac{1}{\tilde{t}})$$
with $C$ depending on $\lambda,M,J,L$ and $T$.

The proof of the inequality for $H^{k+\beta+1-\delta}$ is very similar to that of lemmas 2.9 and 3.8 in \cite{Zoroacordoba}, so we will skip most of the details and focus on the few differences for the sake of briefness.  The idea is to use that

$$\frac{\partial||\Lambda^{s}W||^2_{L^{2}} }{\partial t}\leq 2|\int  (\Lambda^{s}W)\Lambda^{s}(\frac{\partial W}{\partial t})dx|,$$
and then bound each of the integrals obtained from the equation for $\frac{\partial W}{\partial t}$. For example, for the term

$$|\int  (\Lambda^{s}W)\Lambda^{s}(v_{\gamma}(W)\nabla \bar{w}_{\lambda,N,M,J,L,\tilde{t}})dx|$$
we use the Kato-Ponce inequalities obtained in theorem 1.2 of \cite{Dongli} to get for $s=k+\beta+1-\delta$ the inequality

\begin{align*}
    &|\int  (\Lambda^{s}W)\Lambda^{s}(v_{\gamma}(W)\cdot\nabla \bar{w}_{\lambda,N,M,J,L,\tilde{t}})dx|\\
    &\leq \sum_{|\mathbf{a}|\leq s-\gamma }|\int  \frac{1}{\mathbf{a}!}(\Lambda^{s}W)\Lambda^{s,\mathbf{a}}(v_{\gamma}(W))\cdot\nabla \partial^{\mathbf{a}}\bar{w}_{\lambda,N,M,J,L,\tilde{t}})dx|\\
    &+\sum_{|\mathbf{b}|< \gamma }|\int  \frac{1}{\mathbf{b}!}(\Lambda^{s}W)\partial^{\mathbf{b}}(v_{\gamma}(W))\cdot\nabla \Lambda^{s,\mathbf{b}}(\bar{w}_{\lambda,N,M,J,L,\tilde{t}}))dx|\\
    &+C||(\Lambda^{s}W)||_{L^{2}}||v_{\gamma}(W)||_{H^{s-\gamma}}||\Lambda^{\gamma}\nabla \bar{w}_{\lambda,N,M,J,L,\tilde{t}}||_{L^{\infty}},\\
\end{align*}

where we used the multi-index notation, $\mathbf{c}=(c_{1},c_{2})$  , $|\mathbf{c}|=(c_{1}^2+c_{2}^2)^{\frac12}$, $\mathbf{c}!=c_{1}!c_{2}!$, $\partial^{\mathbf{c}}=\partial^{\mathbf{c}}_{x}=\partial^{c_{1}}_{x_{1}}\partial^{c_{2}}_{x_{2}}$ and the operator $\Lambda^{s,\mathbf{c}}$ is defined via the Fourier transform as

$$\widehat{\Lambda^{s,\mathbf{j}}f}(\xi)=\widehat{\Lambda^{s,\mathbf{j}}}(\xi)\hat{f}(\xi)$$
$$\widehat{\Lambda^{s,\mathbf{j}}}(\xi)=i^{-|\mathbf{j}|}\partial^{\mathbf{j}}_{\xi}(|\xi|^s).$$


Most of these terms can be bounded  directly by $C||W||^2_{H^{s}}$  using the properties of $v_{\gamma}$, $\Lambda^{s,\mathbf{c}}$,  and $\bar{w}_{\lambda,N,M,J,L,\tilde{t}}$ plus the assumptions for $W$ (including the $L^2$ growth) and the interpolation inequality for Sobolev spaces.

A few terms, however, requires more careful consideration (and it also needs to be treated differently compared to the proofs in \cite{Zoroacordoba}), namely, for $i=1,2$

\begin{equation}\label{velocidadsingularcota}
    |\int  (\Lambda^{s}W)\Lambda^{s}(v_{i,\gamma}(W)) \frac{\partial \bar{w}_{\lambda,N,M,J,L,\tilde{t}}}{\partial x_{i}})dx|
\end{equation}
\begin{equation*}
    |\int  (\Lambda^{s}W)\Lambda^{s}(v_{i,\gamma}(W)) \frac{\partial W}{\partial x_{i}})dx|
\end{equation*}

since $||\Lambda^{s}(v_{\gamma}(W)||$ cannot by bounded by $||W||_{H^{s}}$. We will just focus on (\ref{velocidadsingularcota}) since the other term is done in exactly the same way. Here, we need to again act as in the $L^2$ case, rewriting (\ref{velocidadsingularcota}) as

$$\frac{1}{2}|\int  (\Lambda^{s}W)\big(v_{i,\gamma}[\Lambda^{s}(W)\frac{\partial  \bar{w}_{\lambda,N,M,J,L,\tilde{t}}}{\partial x_{i}}]-v_{i,\gamma}[\Lambda^{s}(W)]\frac{\partial  \bar{w}_{\lambda,N,M,J,L,\tilde{t}}}{\partial x_{i}}\big)dx|.$$

We can then use again the results obtained in \cite{Dongli} to get
\begin{align*}
    &\frac{1}{2}|\int  (\Lambda^{s}W)\big(v_{i,\gamma}[\Lambda^{s}(W)\frac{\partial  \bar{w}_{\lambda,N,M,J,L,\tilde{t}}}{\partial x_{i}}]-v_{i,\gamma}[\Lambda^{s}(W)]\frac{\partial  \bar{w}_{\lambda,N,M,J,L,\tilde{t}}}{\partial x_{i}}\big)dx|\\
    &\leq C||W||_{H^{s}}||W||_{H^{s}}||v_{i,\gamma}(\frac{\partial \bar{w}_{\lambda,N,M,J,L,\tilde{t}}}{\partial x_{i}})||_{L^{\infty}}\leq C ||W||^2_{H^{s}}.\\
\end{align*}

Combining the bounds for all the terms we obtain

$$\frac{\partial ||\Lambda^{s}W||^2_{L^{2}}}{\partial t}\leq C ||W||_{H^{s}}(||W||_{H^{s}}+(1+\frac{1}{\tilde{t}})\frac{C}{N^\delta})$$
and therefore, for $t\in[0,T]$
$$||W(x,t)||_{H^{s}}\leq Ce^{Ct}t||F||_{H^{s}}\leq C (1+\frac{1}{\tilde{t}})tN^{-\delta}$$
with $C$ depending on $T,\lambda,M,J$ and $L$.
\end{proof}

Combining all the technical results together we obtain.

\begin{theorem}\label{crecimientoperturbado}
Given $T,t_{crit},\epsilon_{1},\epsilon_{2},\epsilon_{3}>0$ and $t_{crit}\in(0,T]$, we can find $\lambda,M,J,L$ and $\tilde{t}$ such that, if $N$ is big enough, then for any $v_{error}$ satisfying
$$||v_{error}||_{C^{m}}\leq \frac{ N^{m}}{N^{k+\beta+2}}$$
for $m=0,1,...,k+2$ and

$$\frac{\partial v_{1,error}}{\partial x_{1}}+\frac{\partial v_{2,error}}{\partial x_{2}}=0$$
then the unique $H^{k+\beta+1-\delta}$ function $\tilde{w}_{\lambda,N,M,J,L,\tilde{t}}(x,t)$ satisfying

\begin{equation}\label{sqgvaprox}
    \frac{\partial \tilde{w}_{\lambda,N,M,J,L,\tilde{t}} }{\partial t}+(v_{\gamma}(\tilde{w}_{\lambda,N,M,J,L,\tilde{t}})+v_{error})\cdot(\nabla\tilde{w}_{\lambda,N,M,J,L,\tilde{t}})=0
\end{equation}
$$\tilde{w}_{\lambda,N,M,J,L,\tilde{t}}(x,0)=\bar{w}_{\lambda,N,M,J,L,\tilde{t}}(x,0)$$
exists for $t\in[0,T]$ and has the following properties.

\begin{itemize}
    \item $||\tilde{w}_{\lambda,N,M,J,L,\tilde{t}}(x,0)||_{C^{k,\beta}}\leq \epsilon_{1}$
    \item $||\tilde{w}_{\lambda,N,M,J,L,\tilde{t}}(x,t)||_{C^{k,\beta}}\geq \frac{1}{\epsilon_{2}}$ if $t\in(t_{crit}-Ct_{crit},t_{crit})$ with $C$ depending on $\epsilon_{1}$ and $\epsilon_{2}$,
    \item $||\tilde{w}_{\lambda,N,M,J,L,\tilde{t}}(x,0)||_{H^{k+\beta+1-\frac32\delta}},||\tilde{w}_{\lambda,N,M,J,L,\tilde{t}}(x,0)||_{L^{1}}\leq \epsilon_{3}$
\end{itemize}
\end{theorem}

\begin{proof}

We first fix some parameters so the pseudo-solutions  $\bar w_{\lambda,N,M,J,L,\tilde{t}}$ have some desirable properties. We fix $\tilde{t}=t_{crit}$ so that, by lemma \ref{crecimientockbeta} we have

$$|\bar{w}_{\lambda,N,M,J,L,\tilde{t}}(x,t_{crit})|_{C^{k,\beta}}\geq \lambda(\frac{1}{4(4\pi)^{\beta}}- -\frac{CL^{2}  }{M}-C(NM)^{-\delta}-CN^{-1+\delta}).$$

Since we want $\tilde{w}$ to also have a very big $C^{k,\beta}$ norm, this suggest taking $\lambda\approx \frac{1}{\epsilon_{2}}$, and we will specifically consider $\lambda=\frac{1}{\epsilon_{2}}32(4\pi)^{\beta}.$

With $\lambda$ fixed, we can now focus on assuring that our initial conditions have a norm as small as required. Using lemmas \ref{normasckpert}, \ref{normackbetapert} and \ref{crecimientockbeta}  plus our choice for $\alpha_{1}^{j}$ we know that


$$||\tilde{w}_{\lambda,N,M,J,L,\tilde{t}}(x,0)||_{C^{k,\beta}}\leq C\lambda_{0}+ C\lambda(\frac{1}{J}+\frac{1}{(NM)^{\delta}}+\frac{J}{ L}+ (NM)^{-\beta}+\frac{L}{M})$$
$$=C\frac{M^{1-\gamma}}{\tilde{t}N^{\gamma}}+ C\lambda(\frac{1}{J}+\frac{1}{(NM)^{\delta}}+\frac{J}{ L}+ (NM)^{-\beta}+\frac{L}{M}).$$
and that there are sets $A$  and $B$ (depending on $\lambda,N,M,J$ and $L$) such that if $x\in A$ then there exists unitary $u$ depending on $x$ with

$$|\frac{\partial^{k}(\bar{w}_{\lambda,N,M,J,L,\tilde{t}}(x,\tilde{t})-\lambda_{0}f_{1})}{\partial u^{k}}|\geq  \lambda( \frac{1}{2(MN)^{\beta}}-\frac{CL^{2}  }{(NM)^{\beta}M}-C(NM)^{-(\delta+\beta)}-C(NM)^{-\beta}N^{-1+\delta} ) $$
and a set $B$ such that if $x\in B$ then for all unitary $u$ we have that
$$|\frac{\partial^{k}(\bar{w}_{\lambda,N,M,J,L,\tilde{t}}(x,\tilde{t})-\lambda_{0}f_{1})}{\partial u^{k}}|\leq  \lambda ( \frac{1 }{4(MN)^{\beta}}+ C(NM)^{-(\delta+\beta)}+\frac{CL^{2}  }{(NM)^{\beta}M})$$
furthermore, there is a set $S_{M,N,\delta}$ such that its cardinal fulfils $|S_{M,N,\delta}|\geq C_{1}MN^{2\delta}$ and

$$A=\cup_{s\in S_{M,N,\delta}} A_{s}$$

$$B=\cup_{s\in S_{M,N,\delta}}B_{s}$$
$d(x,y)\leq \frac{4\pi}{NM}$ if $x\in A_{s}, y\in B_{s}$,  and $|A_{s}|,|B_{s}|\geq \frac{C_{2}}{(NM)^2}$, with $C_{1}$ and $C_{2}$  constants.

By taking $t=t_{crit}$ and $J^2=L$, $M=L^3=J^6$ and  fixing $J$ big we can then obtain that


$$||\tilde{w}_{\lambda,N,M,J,L,\tilde{t}}(x,0)||_{C^{k,\beta}}\leq C\frac{J^{6(1-\gamma)}}{t_{crit}N^{\gamma}}+ \frac{\epsilon_{1}}{2}$$
and for $x\in A$ there exists $u$ unitary such that
\begin{equation}\label{vk}
    |\frac{\partial^{k}(\bar{w}_{\lambda,N,M,J,L,\tilde{t}}(x,\tilde{t})-\lambda_{0}f_{1})}{\partial u^{k}}|\geq  14 \frac{(4\pi)^{\beta} }{\epsilon_{2}(MN)^{\beta}}
\end{equation}
and for $x\in B$ and any unitary vector $u$

$$|\frac{\partial^{k}(\bar{w}_{\lambda,N,M,J,L,\tilde{t}}(x,\tilde{t})-\lambda_{0}f_{1})}{\partial u^{k}}|\leq  10 \frac{(4\pi)^{\beta} }{\epsilon_{2}(MN)^{\beta}}.$$

Note that, the choice of the parameters $J,L$ and $M$ depend only on $\epsilon_{1}$ and $\epsilon_{2}$.

We would like to obtain similar bounds for $\tilde{w}$, so we need to show that $\tilde{w}$ and $\bar{w}$ are close to each other in a useful way.. First, using lemma \ref{evolucionerror} we have

$$\sum_{i=0}^{k}\int (\frac{\partial^{k}(\tilde{w}-\bar{w})}{\partial x_{1}^{i}\partial x_{2}^{k-i}})^2\leq C (1+\frac{1}{\tilde{t}})N^{-2(\beta+1)}$$
and in particular (including from now on $(1+\frac{1}{\tilde{t}})$ inside of the constant $C$ since  it is constant with respect to $N$), there exists $A_{s}, B_{s}$ such that

$$\sum_{i=0}^{k}\int_{A_{s}} (\frac{\partial^{k}(\tilde{w}-\bar{w})}{\partial x_{1}^{i}\partial x_{2}^{k-i}})^2+\int_{B_{s}} (\frac{\partial^{k}(\tilde{w}-\bar{w})}{\partial x_{1}^{i}\partial x_{2}^{k-i}})^2\leq C N^{-2(\beta+1+\delta)}$$
so

$$inf_{x\in A_{s}}|\sum_{i=0}^{k}(\frac{\partial^{k}(\tilde{w}-\bar{w})}{\partial x_{1}^{i}\partial x_{2}^{k-i}})^2||A_{s}|\leq\sum_{i=0}^{k} \int_{A_{s}} (\frac{\partial^{k}(\tilde{w}-\bar{w})}{\partial x_{1}^{i}\partial x_{2}^{k-i}})^2\leq C N^{-2(\beta+1+\delta)}$$

$$inf_{x\in B_{s}}|\sum_{i=0}^{k}(\frac{\partial^{k}(\tilde{w}-\bar{w})}{\partial x_{1}^{i}\partial x_{2}^{k-i}})^2||B_{s}|\leq \sum_{i=0}^{k}\int_{B_{s}} (\frac{\partial^{k}(\tilde{w}-\bar{w})}{\partial x_{1}^{i}\partial x_{2}^{k-i}})^2\leq C N^{-2(\beta+\delta+1)}$$
and therefore

$$inf_{x\in A_{s}}|\sum_{i=0}^{k}(\frac{\partial^{k}(\tilde{w}-\bar{w})(x,t)}{\partial x_{1}^{i}\partial x_{2}^{k-i}})^2|\leq C N^{-2(\beta+\delta)}$$

$$inf_{x\in B_{s}}|\sum_{i=0}^{k}(\frac{\partial^{k}(\tilde{w}-\bar{w})(x,t)}{\partial x_{1}^{i}\partial x_{2}^{k-i}})^2|\leq C N^{-2(\beta+\delta)}.$$

Given a time $t\in[0,t_{crit}]$, we dconsider $x_{A}(t)\in A_{s}$, $x_{B}(t)\in B_{s}$ points fulfilling

$$|\sum_{i=0}^{k}(\frac{\partial^{k}(\tilde{w}-\bar{w})(x_{A}(t),t)}{\partial x_{1}^{i}\partial x_{2}^{k-i}})^2|\leq C N^{-2(\beta+\delta)}$$

$$|\sum_{i=0}^{k}(\frac{\partial^{k}(\tilde{w}-\bar{w})(x_{B}(t),t)}{\partial x_{1}^{i}\partial x_{2}^{k-i}})^2|\leq C N^{-2(\beta+\delta)}.$$

Now, if $u$ is the unitary vector given by (\ref{vk}) for $x_{B}(t)$, we have that

\begin{align*}
    &\big|\frac{\partial^{k}\tilde{w}(x_{A},t)-\tilde{w}(x_{B},t)}{\partial^{k}u}\big|\frac{1}{|x_{A}-x_{B}|^{\beta}}\\
    &\geq \big|\frac{\partial^{k}\bar{w}(x_{A},t)-\bar{w}(x_{B},t)}{\partial^{k}u}\big|\frac{1}{|x_{A}-x_{B}|^{\beta}}-C N^{-\delta}\\
    &\geq \big|\frac{\partial^{k}\bar{w}(x_{A},t_{crit})-\bar{w}(x_{B},t_{crit})}{\partial^{k}u}\big|\frac{1}{|x_{A}-x_{B}|^{\beta}}-||\bar{w}(x,t)-\bar{w}(x,t_{crit})||_{C^{k,\beta}}-C N^{-\delta} \\
    &\geq \frac{4}{\epsilon_{2}}-C\frac{\lambda J^{6}|t-\tilde{t}|}{\tilde{t}}-CN^{-\delta}\\
\end{align*}

where we used lemma \ref{derivadackbeta} in the last inequality. Then if $|C\frac{\lambda J^{6}|t-\tilde{t}|}{\tilde{t}}|\leq \frac{2}{\epsilon_{2}}$, $|CN^{-\delta}|\leq \frac{1}{\epsilon_{2}}$ we get

\begin{align*}
    &||\tilde{w}(x,t)||_{C^{k,\beta}}\geq \big|\frac{\partial^{k}\tilde{w}(x_{A},t)-\tilde{w}(x_{B},t)}{\partial^{k}v}\big|\frac{1}{|x_{A}-x_{B}|^{\beta}}\geq \frac{1}{\epsilon_{2}}\\
\end{align*}
and this will be true if we take $N$ big enough and $|t-\tilde{t}|\leq \frac{\tilde{t}}{\lambda J^{6}|t-\tilde{t}|}=C(\epsilon_{1},\epsilon_{2})$.

The only thing we need to prove is that we can also obtain 
$$||\tilde{w}_{\lambda,N,M,J,L,\tilde{t}}(x,0)||_{C^{k,\beta}}\leq \epsilon_{1}$$
$$||\tilde{w}(x,0)||_{H^{k+\beta+1-\frac32\delta}}\leq \epsilon_{3},$$
but
$$||\tilde{w}(x,0)||_{H^{k+\beta+1-\frac32\delta}}\leq \frac{C}{N^{\gamma}}+\frac{C}{N^{\frac{\delta}{2}}}$$
with $C$ depending on $J,L$ and $M$, so taking $N$ big enough
$$||\tilde{w}(x,0)||_{H^{k+\beta+1-\frac32\delta}}\leq \epsilon_{3}$$
and analogously,
$$||\tilde{w}_{\lambda,N,M,J,L,\tilde{t}}(x,0)||_{C^{k,\beta}}\leq C\frac{J^{6(1-\gamma)}}{t_{crit}N^{\gamma}}+ \frac{\epsilon_{1}}{2}$$
so again, taking $N$ big enough finishes the proof.



\end{proof}

\section{Strong ill-posedness and non-existence of solutions}

We are now ready to prove ill-posedness and non-existence of solutions. As mentioned in section 4, these results hold for $k\in\mathds{N}$, $\beta\in(0,1]$, $\gamma\in(0,1)$ with $k+\beta> 1+\gamma$ and $\delta$ is some constant $\delta\in(0,\frac12)$ such that $k+\beta+2 \delta > 1+\gamma$.

\begin{theorem}
Given $T,t_{crit,}\epsilon_{1},\epsilon_{2}>0$, there exist a function $w(x,0)$ such that $||w(x,0)||_{C^{k,\beta}}\leq \epsilon_{1}$ and  the only solution to (\ref{gSQG}) in $H^{k+\beta+1- \delta}$ with initial conditions $w(x,0)$ exists for $t\in[0,T]$ and fulfills that
$$||w(x,t_{crit})||_{C^{k,\beta}}\geq \frac{1}{\epsilon_{2}}.$$
\end{theorem}

\begin{proof}
This is just a direct application of theorem \ref{crecimientoperturbado} with $v_{1,error}=v_{2,error}=0$.
\end{proof}

\begin{theorem}
Given $t_{0},\epsilon>0$, there exist a function $w(x,0)$ such that $||w(x,0)||_{C^{k,\beta}}\leq \epsilon$ and that the only solution to  (\ref{gSQG}) in $H^{k+\beta+1-\frac32 \delta}$ with initial conditions $w(x,0)$ exists for $t\in[0,t_{0}]$ and fulfills that, for $t\in(0,t_{0}]$, $||w(x,t)||_{C^{k,\beta}}=\infty.$
\end{theorem}

\begin{proof}
To obtain initial conditions with the desired properties, we will consider initial conditions of the form

$$\sum_{j=1}^{\infty}\sum_{i=1}^{G(j,\epsilon)}T_{R_{i,j}}(w_{i,j}(x))$$

where $T_{R}(f(x_{1},x_{2}))=f(x_{1}+R,x_{2})$. We will first choose $w_{i,j}(x)$ and afterwards we will pick the values of $R_{i,j}$.

First, fixed $j$, we will restrict to choices for $w_{i,j}$ such that they are initial conditions given by theorem \ref{crecimientoperturbado} with $\frac{1}{\epsilon_{2}}=j$, $\epsilon_{1}=\epsilon$ and $T=t_{0}$. Then if we choose some $t_{crit}=t_{crit,i,j}$ and  we call $\tilde{w}_{i,j}$ a solution to (\ref{gSQGvext})  with the initial conditions given by $w_{i,j}(x)$  and an appropriate $v_{ext}$ fulfilling $||v_{ext}||_{C^{k+2}}\leq C_{i,j}$, we would then have that for $t\in[t_{crit,i,j}-Ct_{crit,i,j},t_{crit,i,j}]$

$$||\tilde{w}_{i,j}(x,t)||_{C^{k,\beta}}\geq j$$
for some $C$ depending on $\epsilon$ and $j$. Therefore, we can, by choosing $t_{crit,i,j}$ appropriately, obtain, for any $t\in[\frac{1}{j},t_{0}]$

$$sup_{i=1,2,...,G(j,\epsilon)}||\tilde{w}_{i,j}(x,t)||_{C^{k\beta}}\geq j$$
with $G(j,\epsilon)$ a finite number depending on $j$.

Furthermore, we can now choose $\epsilon_{3}$ in theorem \ref{crecimientoperturbado} so that

$$||w_{i,j}(x)||_{H^{k+\beta+1-\frac32 \delta}}\leq \frac{c_{0}2^{-j}}{G(j,\epsilon)}$$
$$||w_{i,j}(x)||_{L^{1}}\leq \frac{2^{-j}}{G(j,\epsilon)}$$
with $c_{0}$ a constant small enough so that any solution to $\gamma$-SQG with $||w_{0}(x)||_{H^{k+\beta+1-\frac32 \delta}}\leq c_{0}$ exists for $t\in[0,t_{0}]$ and $||w(x,t)||_{H^{k+\beta+1-\frac32 \delta}}\leq 1$ for $t\in[0,t_{0}]$.
Therefore we know that, independently of the choice of $R_{i,j}$, for $t\in[0,t_{0}]$ there exists a unique $H^{k+\beta+1-\frac32\delta}$ solution to (\ref{gSQG}) with initial conditions,
$\sum_{j=1}^{\infty}\sum_{i=1}^{G(j)}T_{R_{i,j}}(w_{i,j}(x))$  and, furthermore, if we call this solution $w_{\infty}(x,t)$ (which still depends on the choice of $R_{i,j}$, but we omit it for simplicity of notation), then we have that there is a constant $v_{max}$ such that, for $t\in[0,t_{0}]$

$$||v_{1}(w_{\infty})||_{L^{\infty}},||v_{2}(w_{\infty})||_{L^{\infty}}\leq v_{max}.$$
With this, and using that there exists $D\in\mathds{R}$ such that $supp(w_{i,j}(x))\subset B_{D}(0)$, we have that, if we choose  the $R_{i,j}$ so that $|R_{i_{1},j_{1}}-R_{i_{2},j_{2}}|\geq 4t_{0} v_{max}+2D+sup(P_{i_{1},j_{1}},P_{i_{2},j_{2}})$ with $P_{i,j}>0$ then we have that  $$w_{i,j,\infty}(x,t):=1_{B_{D+2t_{0}v_{max}}(-R_{i,j},0)}w_{\infty}(x,t)$$
fulfils for $t\in[0,t_{0}]$ the evolution equation

$$\frac{\partial w_{i,j,\infty} }{\partial t}+(v_{\gamma}(w_{i,j,\infty})+v(w_{\infty}-w_{i,j,\infty}))\cdot(\nabla w_{i,j,\infty})=0$$
and
$$||v(w_{\infty}-w_{i,j,\infty}))||_{C^{k+2}}\leq \frac{C}{P_{i,j}^{2+\gamma}}.$$

But by the choice of $w_{i,j}(x)$ and using that the supports of the $w_{i,j\infty}$ are disjoint, we have that if
\begin{equation}\label{cij}
    ||v(w_{\infty}-w_{i,j,\infty}))||_{C^{k+2}}\leq C_{i,j}
\end{equation}
then for $t\in(0,t_{0}]$

\begin{align*}
    &||w_{\infty}(x,t)||_{C^{k,\beta}}= \text{sup}_{j\in\mathds{N},i=1,2,...,G(j,\epsilon)}||w_{i,j,\infty}(x,t)||_{C^{k,\beta}}=\infty\\
\end{align*}
and taking $P_{i,j}$ big enough so that (\ref{cij}) is fulfilled finishes the proof.

\end{proof}
\section*{Acknowledgements}
This work is supported in part by the Spanish Ministry of Science
and Innovation, through the “Severo Ochoa Programme for Centres of Excellence in R$\&$D
(CEX2019-000904-S)” and 114703GB-100. DC and LMZ were partially supported by the ERC Advanced Grant 788250. DC gratefully acknowledges the support of the Charles Simonyi Endowment at the Institute for Advanced Study (Princeton).

\bibliographystyle{alpha}

\end{document}